\documentclass{amsart}
\usepackage{graphicx} 
\usepackage{tikz} 

\usepackage{amsmath}
\usepackage{amsthm}
\usepackage{amssymb}
\usepackage{amsfonts}
\usepackage[
backend=biber,
style=ieee,
sorting=nyt
]{biblatex}
\newtheorem{theorem}{Theorem}[section]
\newtheorem{definition}[theorem]{Definition}
\newtheorem{prop}[theorem]{Proposition}
\newtheorem{coroll}[theorem]{Corollary}
\newtheorem{lemma}[theorem]{Lemma}
\newtheorem*{claim}{Claim}

\numberwithin{equation}{section}

\newtheorem*{thm:countsubgroup}{Theorem \ref{thm:countingsubgrps}}
\newtheorem*{thm:limitmeasure}{Theorem \ref{thm:limitmeasure}}
\newtheorem*{thm:pattsulconverge}{Theorem \ref{thm:pattsulconverge}}
\newtheorem*{coroll:systole}{Corollary \ref{coroll:systole}}
\newtheorem*{coroll:toptype}{Corollary \ref{coroll:toptype}}
\newtheorem*{coroll:orthgeo}{Corollary \ref{coroll:orthgeo}}

\newcommand{\N}{\mathbb{N}}
\newcommand{\R}{\mathbb{R}}
\newcommand{\Z}{\mathbb{Z}}
\newcommand{\hyp}{\mathbb{H}^2}

\newcommand{\psl}{\operatorname{PSL}_2(\mathbb{R})}

\newcommand{\reps}{\mathcal{X}_{k}}
\newcommand{\critreps}{\mathcal{X}_{\Sigma,k}}
\newcommand{\hreps}{\hat{\mathcal{X}}_{k}}
\newcommand{\cv}{CV_k}
\newcommand{\out}{\operatorname{Out}(\free)}
\newcommand{\aut}{\operatorname{Aut}(\free)}
\newcommand{\gmoduli}{\mathcal{M}_k}
\newcommand{\free}{F_k}
\newcommand{\fund}{\pi_1(\Sigma)}
\newcommand{\gfund}{\pi_1(X)}
\newcommand{\gaut}{\operatorname{Aut}(X)}
\newcommand{\gtriv}{\operatorname{Triv}(X)}
\newcommand{\typ}{\operatorname{Typ_k}}
\newcommand{\critgraph}{\mathbf{Crit}_{\Sigma,k}}
\newcommand{\critx}{\mathbf{Crit}_{X,\Sigma}}

\newcommand{\critAl}{\mathbf{Crit}^{\ell_0}_{A}}
\newcommand{\surfsubgrp}{\mathbf{G}_{\Sigma,k}}
\newcommand{\surfsubgrpx}{\mathbf{G}_{X,\Sigma}}
\newcommand{\subgrps}{\mathcal{G}_{k}}
\newcommand{\hsubgrps}{\hat{\mathcal{G}}_{k}}
\newcommand{\measl}{\mathbf{m}_{\Sigma,k}^L}
\newcommand{\meas}{\mathbf{m}_{k}}
\newcommand{\gmeasl}{\mu_{\Sigma,k}^L}
\newcommand{\gmeaslx}{\mu_{X,\Sigma}^L}
\newcommand{\gmeasx}{\mu_{X}}
\newcommand{\psmeas}{\overline{\mathbf{m}}_{\Sigma,k}^s}
\newcommand{\projc}{\mathbb{RP}^\mathcal{C}}
\newcommand{\sys}{\operatorname{sys}(\Delta)}
\newcommand{\cover}{\Sigma_\Delta}
\newcommand{\cov}{\mathbf{Cov}_{\Sigma,k}}

\title{Random covers of surfaces}
\author{Sophie Wright}
\address{School of Mathematics, University of Bristol, Bristol BS8 1UG, UK}
\email{sophie.wright@bristol.ac.uk}
    
\begin{document}

\begin{abstract}
    We study random covers of a closed hyperbolic surface $\Sigma$, subject to the condition that, for $k\geq 2$, the fundamental group is isomorphic to the free group $F_k$.
    We show that asymptotically they distribute according to a specific probability measure on the moduli space of metric graphs. 
    As we will demonstrate with explicit calculations for $k=2$, this allows us to determine asymptotic values for the expectation of the systole and other geometric invariants of the covers.
\end{abstract}
\maketitle

\section{Introduction}\label{sec:intro}
Fix a closed, connected, orientable hyperbolic surface $\Sigma$ of genus $g\geq 2$. We are interested in the geometric properties of random covers of the surface, that is in random conjugacy classes of subgroups of $\fund$. Any infinite index subgroup of a surface group is free. Fixing $k\geq 2$ we will focus on random subgroups isomorphic to the free group $\free$ of rank $k$. As we are not giving the surface $\Sigma$ a basepoint we will consider these subgroups up to conjugacy in $\fund$. We denote by
\begin{equation*}
    \surfsubgrp:=\left\{\Delta\subset\fund \mid \Delta\cong\free\right\}/\sim
\end{equation*}
the set of conjugacy classes of rank $k$ free subgroups of $\fund$.
Abusing notation slightly, we will let $\Delta$ denote both a subgroup and its respective conjugacy class interchangeably.
The above set is infinite, thus a starting point of this paper is to define what it means to pick a subgroup at random. 

For $k=1$, the set of conjugacy classes of cyclic subgroups of $\fund$ is in a 1-to-1 correspondence with the set of closed unoriented geodesics on the surface. To pick a random geodesic we can, for $L>0$, pick from the finite set of all closed geodesics of length at most $L$, letting then $L\to\infty$. A classical result of Huber \cite{Hub61} gives us the asymptotic growth of the number of geodesics as $L\to\infty$: 
\begin{equation}\label{eq:Huber}
    \left|\left\{
    \begin{array}{cc} \gamma\text{ closed unoriented geodesic}  \\
        \text{with length }\ell_{\Sigma}(\gamma)\leq L
    \end{array}\right\}\right| \sim \frac{e^L}{2L}
\end{equation}
where here, and for the rest of the paper, $\sim$ indicates that the ratio between the two sides tends to 1 as $L\to\infty$.

Our first result is an analogue of Huber's theorem for subgroups of rank $k$ where $k\geq2$. We must first define the length of an element of $\surfsubgrp$. 
We say a graph $X$ with a map $\phi:X\rightarrow\Sigma$ carries $\Delta \in\surfsubgrp$ when the induced homomorphism $\phi_*$ is injective and $\phi_*(\gfund)=\Delta$. This is an adaption of the notion of a carrier graph first introduced for hyperbolic 3-manifolds by White \cite{White02}. The length $\ell_\Sigma(X,\phi)$ of a carrier graph is the sum of the lengths of the images of edges and we define the length of $\Delta\in\surfsubgrp$ to be the minimum of the lengths of all graphs which carry it:
\begin{equation*}
    \ell(\Delta)=\min\left\{\ell_\Sigma(X,\phi)\mid (X,\phi) \text{ carries } \Delta\right\} \, .
\end{equation*}
A result essentially due to White \cite{White02} shows that such a minimal length carrier graph exists and is always trivalent, (see Lemmas \ref{lem:mingraphexists} and \ref{lem:minarecrit}).

We denote by $\surfsubgrp(L)$ the subset of $\surfsubgrp$ containing conjugacy classes of groups with length at most $L$. This set is finite for every $L$ and our first result gives the asymptotic growth of its cardinality:
\begin{theorem}\label{thm:countingsubgrps}
    For a closed, orientable hyperbolic surface $\Sigma$ of genus $g$ and for any $k\geq 2$, the cardinality of the set $\surfsubgrp(L)$ has asymptotic growth:
    \begin{equation*}
        \left|\surfsubgrp(L)\right|\sim \sum_{X\in\typ}\frac{1}{\left|\gaut\right|} \cdot\left(\frac{4}{3}\right)^{3-3k}\cdot\frac{(\pi^2(g-1))^{1-k}}{(3k-4)!}\cdot L^{3k-4}\cdot e^L
    \end{equation*}
    as $L\to\infty$, where $\typ$ is the set of homeomorphism classes of connected trivalent graphs of rank $k$ and $\gaut$ is the set of automorphisms of the topological graph $X$.
\end{theorem}
Note that for each $k$ there are a finite number of topological types of trivalent graphs with rank $k$, each with a finite number of automorphisms. Therefore the sum in the above theorem is finite.

While Theorem \ref{thm:countingsubgrps} is an analogue of Huber's geometric prime number theorem described in (\ref{eq:Huber}) above, work of Sasaki \cite{Sas25} gives the asymptotic growth of the number of mapping class group orbits of a finitely generated subgroup of the fundamental group of a hyperbolic surface, an analogue of Mirzakhani's work on closed geodesics of fixed type \cite{Mirz2016}.

\medskip
Returning now to the notion of a random subgroup, we choose uniformly at random from the finite set $\surfsubgrp(L)$ and then let $L\to\infty$. More precisely, we will study the family of probability measures 
\begin{equation*}
    \measl := \frac{1}{|\surfsubgrp(L)|}\sum_{\Delta\in\surfsubgrp(L)}\delta_{\Delta}
\end{equation*}
where $\delta_{\Delta}$ is the Dirac measure. Identifying the universal cover $\tilde{\Sigma}$ with $\hyp$ and the fundamental group $\fund$ with a subgroup $\Gamma\subset\psl$ with $\Sigma=\hyp/\Gamma$, we consider each $\measl$ to be a measure on:
\begin{equation*}
    \subgrps:=\{\Lambda\subset\psl \text{ discrete subgroup with } \Lambda\cong\free\}/\psl \, .
\end{equation*}
This set is endowed with a natural topology (see Section \ref{sec:spaces}). Therefore $\surfsubgrp(L)$ is a finite set of points in the space $\subgrps$ on which $\measl$ is supported. The main result of this paper describes the limiting behaviour of these measures. Letting $\gmoduli$ be the moduli space of volume one metric graphs of rank $k$ we will equip the union  $\hsubgrps:=\subgrps\cup\gmoduli$ with a suitable structure such that the following holds:

\begin{theorem}\label{thm:limitmeasure}
    The measures $\measl$ on $\hsubgrps$ converge with respect to the weak-*-topology to a limit measure
    \begin{equation*}
        \meas= \lim_{L\to\infty}\measl
    \end{equation*}
    supported on the moduli space of (volume one) metric graphs $\gmoduli$. This limit is a probability measure of full support in the Lebesgue class which is moreover independent of the surface $\Sigma$. 
\end{theorem}

Note that the moduli space of metric graphs can be viewed as a finite piecewise-Euclidean simplicial complex with orbifold identifications, whose open simplices correspond to topological types of graphs. The limit measure is proportional to the standard Lebesgue measure on each simplex, with proportion factor determined by the size of the automorphism group of the underlying graph for the simplex.
We will describe this limit in more detail in Section \ref{sec:measures}.

As an alternative to the counting measures $\measl$, Kaimanovich suggested to instead consider a variant of the classical Patterson-Sullivan measures.
Concretely, for any $s>1$ we define a probability measure on $\hsubgrps$ by: 
\begin{equation}
    \psmeas= \frac{1}{\sum_{\Delta\in \surfsubgrp}e^{-s\ell_{\Sigma}(\Delta)}}\sum_{\Delta\in\surfsubgrp}e^{-s\ell(\Delta)}\delta_\Delta \, .
\end{equation}
It turns out that as $s$ approaches $1$ this new family of measures has the same limit measure as that of the counting measures above: 
\begin{theorem}\label{thm:pattsulconverge}
    The family of measures $\psmeas$ on $\hsubgrps$ converge in the weak-*-topology to the limit measure $\meas$ as $s\downarrow1$.
\end{theorem}

Returning to the prior setting, Theorem \ref{thm:limitmeasure} allows us to examine geometric and topological properties of random covers of the surface $\Sigma$. For example, asymptotically the expected systole of a cover corresponding to a subgroup of rank $k$ will correspond to the $\meas$-expected systole of a metric graph in the moduli space $\gmoduli$. We can make explicit calculations. For example when $k=2$ we have the following: 
\begin{coroll:systole}
    The expected value of the systole of a cover corresponding to a conjugacy class of a free rank $2$ subgroup of $\fund$ of length at most $L$ is asymptotic to $\frac{23}{90}L$ as $L\to\infty$. 
\end{coroll:systole}

Similarly, we can study expected properties of orthogeodesics. An orthogeodesic is a geodesic arc joining boundary components which meet the boundary perpendicularly. An orthogeodesic is separating if it splits the surface into two connected components. When $k=2$ we have the following:
\begin{coroll:orthgeo}
    For any $\lambda>0$, the probability that the convex core of a cover corresponding to an element $\Delta\in\mathbf{G}_{\Sigma,2}(L)$ has a separating orthogeodesic of length less than $\lambda$ tends to $\frac{3}{5}$ as $L\to\infty$.
\end{coroll:orthgeo}

Topologically, by studying covers given by different trivalent graphs we can make predictions on the expected topological type of a rank $k$ cover of a surface. For example when $k=2$ there are only two topological types of covers: the pair of pants and the one-holed torus. We calculate a lower bound on the probability that a cover is a pair of pants:
\begin{coroll:toptype}
    The probability that a cover corresponding to a rank $2$ subgroup of length at most $L$ is a pair of pants is greater than or equal to $\frac{3}{5}$ as $L\to\infty$.
\end{coroll:toptype}
In a forthcoming paper \cite{Wri} we will study in more detail the topological properties of random covers. This will allow us to determine the precise limiting probabilities for the expected topological type of a cover and calculate an exact value for the statement in Corollary \ref{coroll:toptype} above.

\subsection*{A note on other models} We observe that using minimal length carrier graphs to construct probability measures for subgroups allows us to see different geometric details when compared to other possible models. For instance, one could choose a random subgroup of rank $k$ by picking $k$ generators each with translation length less than $L$. Letting the generator lengths grow to infinity one would get that the limit is the Dirac measure on the rose $R_k$ with $k$ petals of equal length. By instead using carrier graphs to study subgroups we give these groups a geometric structure independent of any choice of generators. 

\subsection*{Strategy of the proofs} The key idea in the proof of Theorem \ref{thm:countingsubgrps} is the observation that generically, as $L\to\infty$, subgroups have a unique minimal carrier graph. 
A result of Erlandsson and Souto \cite{ES23-commutator} gives the asymptotic growth of the number of minimal length immersions of a trivalent graph of a fixed topological type into a surface. We extend this result in order to calculate the limiting growth of the number of rank $k$ minimal carrier graphs on the surface and this gives us a growth rate for number of conjugacy classes of free rank $k$ subgroups of $\fund$.

To prove Theorem \ref{thm:limitmeasure} we use that the number of minimal length immersions of a trivalent graph equidistributes over the simplex corresponding to the edge lengths of the graph \cite{ES23-commutator}. We extend this equidistribution result to minimal length carrier graphs to see that the graphs spread evenly within each top-dimensional simplex in the moduli space $\gmoduli$. The elements of $\surfsubgrp$ get arbitrarily close to points in the moduli space corresponding to their carrier graphs and thus both have the same limit measure. 

\subsection*{Section-by-section summary} 
In Section \ref{sec:spaces} we discuss the topology of the space $\hsubgrps=\subgrps\cup\gmoduli$ of conjugacy classes of subgroups and its bordification by the moduli space of metric graphs.

In Section \ref{sec:graphs} we discuss lengths of subgroups defined via lengths of carrier graphs. We then prove Theorem \ref{thm:countingsubgrps} which gives the asymptotic growth of the number of conjugacy classes of subgroups.

In Section \ref{sec:measures} we introduce the probability counting measures $\measl$ on the space $\hsubgrps$ and prove Theorem \ref{thm:limitmeasure} showing that they converge to a measure $\meas$ in the Lebesgue class supported on $\gmoduli$.

In Section \ref{sec:patt-sull} we discuss the Patterson-Sullivan style probability measures $\psmeas$ on the space $\hsubgrps$ and show that these converge to the same limit measure $\meas$ from the previous case.

In Section \ref{sec:expectedprop} we discuss how to use these measures to study geometric properties of covers and we prove Corollaries \ref{coroll:systole}, \ref{coroll:orthgeo} and \ref{coroll:toptype}. 

\subsection*{Acknowledgements} 
The author would like to thank Viveka Erlandsson and Juan Souto for suggesting the project and for their invaluable help and advice throughout. The idea behind Section \ref{sec:patt-sull} on a variant of Patterson-Sullivan measure arose from discussions with Vadim Kaimanovich and Mingkun Liu. 

The author acknowledges the support of the Institut Henri Poincar\'e (UAR 839 CNRS-Sorbonne Universit\'e) and LabEx CARMIN (ANR-10-LABX-59-01), the College Doctoral de Bretagne, the University of Rennes, Rennes M\'etropole and the Heilbronn Institute for Mathematical Research.

\section{Spaces of conjugacy classes of subgroups}\label{sec:spaces}
The goal of this section is to describe a bordification of the space $\subgrps$ of conjugacy classes of discrete rank $k$ free subgroups of $\psl$ by the moduli space of metric graphs. It will be important that the obtained space is locally compact.

Given our closed hyperbolic orientable surface $\Sigma$ and some $k\geq2$ we denote, as in the introduction:
\begin{equation*}
    \surfsubgrp=\left\{\Delta\subset\fund \mid \Delta\cong\free \right\}/\sim
\end{equation*}
the set of conjugacy classes of subgroups of $\fund$ isomorphic to $\free$. 
As mentioned above, we will abuse notation and write $\Delta$ to represent both a subgroup of $\fund$ and its conjugacy class.

We will consider $\surfsubgrp$ as living inside a larger ambient space. To that end, we identify the universal cover of $\Sigma$ with the hyperbolic plane $\hyp$ and the fundamental group $\fund$ with a discrete subgroup of $\psl$, the isometry group of $\hyp$. 
We can then consider elements of $\surfsubgrp$ to be points in the set:
\begin{equation*}
    \subgrps=\{\Lambda\subset\psl \text{ discrete subgroup with }\Lambda\cong\free\}/\sim
\end{equation*}
where two subgroups are equivalent if they are conjugate in $\psl$. Note that the map $\surfsubgrp\to\subgrps$ is not injective, due to the fact that $\Sigma$ always has distinct covers which are isometric. We will not be bothered by this fact.

We want to view the set $\subgrps$ as a moduli space for free subgroups of $\psl$, analogous to the moduli space of hyperbolic surfaces, however we first need to endow it with a suitable topology and structure.

For each element $\Lambda\in\subgrps$ we can add a marking by specifying the isomorphism $\free\rightarrow\Lambda$. This is equivalent to fixing a choice of $k$ generators for the subgroup $\Lambda$. The set of marked elements of $\subgrps$ is given by:
\begin{equation*}
    \reps=\{\rho:\free\rightarrow\psl \mid \text{faithful, discrete homomorphism}\}/\sim
\end{equation*}
where the equivalence relation is conjugation in $\psl$. This set is part of the $\psl$-character variety of $\free$.
We endow $\reps$ with the induced topology, that is the topology coming from length functions of representations.
Under the usual action of $\psl$ on $\hyp$ and for $\rho\in\reps$ we can assign to each $\alpha\in\free$ a length corresponding to the translation lengths in $\hyp$ of their images under $\rho$: 
\begin{equation*}
    \ell_{\rho}:\free \rightarrow\R_{\geq0}, \quad
    \alpha \mapsto\inf_{x\in\hyp}\{d_{\hyp}(x,\rho(\alpha)\cdot x)\} \, .
\end{equation*}
We note that translation lengths are invariant under conjugation and therefore $\rho\mapsto\ell_\rho(\cdot)$ defines a map $\reps\rightarrow\R_{\geq0}^{\mathcal{C}}$ where $\mathcal{C}$ is the set of conjugacy classes of the free group $\free$.
In fact, every discrete faithful representation $\rho:\free\rightarrow\psl$ is uniquely determined up to conjugacy by its length function \cite{CM87}, thus the above map is injective and we have an embedding $\reps\hookrightarrow\R_{\geq0}^{\mathcal{C}}$. 

This embedding induces a map $\reps\rightarrow\projc$ into the projective space. It is well-known in the setting of representations of closed surfaces that the map from length functions to projective space is injective, and thus is an embedding. 
The same is true in the case of discrete faithful representations of free groups into $\psl$, these being representations corresponding to non-closed surfaces. In this setting the argument becomes more involved, we give a proof here utilising trace relations for hyperbolic elements of $\psl$: 
\begin{lemma}\label{lem:embedtoproj}
    The embedding $\reps\hookrightarrow\R_{\geq0}^{\mathcal{C}}$ descends to an embedding $\reps\hookrightarrow\projc$ into the projectivise space.
\end{lemma}
\begin{proof}
Suppose representations $\rho_1,\rho_2\in\reps$ have length functions differing by a scalar $\lambda_0>0$, so that $\ell_{\rho_2}(\alpha)=\lambda_0\ell_{\rho_1}(\alpha)$ for all $\alpha\in\free$. We can assume, without loss of generality, that $\lambda_0\geq1$. We will show that we must have $\lambda_0=1$.
 
For hyperbolic elements of $\psl$ we have an explicit relationship between translation length and trace:
\begin{equation*}
    \ell_{\hyp}(A)=2\operatorname{arcosh}\left(\frac{|\operatorname{tr}(A)|}{2}\right)
\end{equation*}
for all $A\in\psl$ hyperbolic. Given $\alpha,\beta\in\free$ where  $\rho_i(\alpha)$ and $\rho_i(\beta)$ are hyperbolic, and using the trace identity $\operatorname{tr}(A)\operatorname{tr}(B)=\operatorname{tr}(AB)+\operatorname{tr}(AB^{-1})$ for $A,B\in\psl$, we have two cases, either
\begin{equation}\label{eq:embedineqpos}
    2\cosh\left(\frac{\ell_{\rho_{i}}(\alpha)}{2}\right)\cosh\left(\frac{\ell_{\rho_{i}}(\beta)}{2}\right)=\cosh\left(\frac{\ell_{\rho_{i}}(\alpha\beta)}{2}\right)+\cosh\left(\frac{\ell_{\rho_{i}}(\alpha\beta^{-1})}{2}\right)
\end{equation}
when $\operatorname{tr}(\rho_i(\alpha\beta))$ and $\operatorname{tr}(\rho_i(\alpha\beta^{-1}))$ have the same sign, or
\begin{equation}\label{eq:embedineqneg}
    2\cosh\left(\frac{\ell_{\rho_{i}}(\alpha)}{2}\right)\cosh\left(\frac{\ell_{\rho_{i}}(\beta)}{2}\right)=\left|\cosh\left(\frac{\ell_{\rho_{i}}(\alpha\beta)}{2}\right)-\cosh\left(\frac{\ell_{\rho_{i}}(\alpha\beta^{-1})}{2}\right)\right|
\end{equation}
when $\operatorname{tr}(\rho_i(\alpha\beta))$ and $\operatorname{tr}(\rho_i(\alpha\beta^{-1}))$ have different signs.

Select an element $\alpha\in\free$ with $\rho_1(\alpha)$ hyperbolic and let $\beta=\gamma^n\alpha\gamma^{-n}$ for some $\gamma\in\free$ and large enough $n$ so that the translation axes of $\rho_1(\alpha)$ and $\rho_1(\beta)$ do not intersect and are at least distance $D>0$ apart. When we make $D$ large enough we get strict inequalities: 
\begin{equation*}
    \begin{split}
        \ell_{\rho_1}(\alpha\beta)&>\ell_{\rho_1}(\alpha)+\ell_{\rho_1}(\beta) \\
        \ell_{\rho_1}(\alpha\beta^{-1})&>\ell_{\rho_1}(\alpha)+\ell_{\rho_1}(\beta) \, .
    \end{split}
\end{equation*}
Note that as $\alpha$ and $\beta$ are conjugate they have the same translation length under $\rho_1$ and we will denote $a=\frac{\ell_{\rho_1}(\alpha)}{2}=\frac{\ell_{\rho_1}(\beta)}{2}$. Given the above inequalities we see that the right hand side of (\ref{eq:embedineqpos}) is strictly greater than the left. Therefore we are in the second case with equation (\ref{eq:embedineqneg}).
 
Denote:
\begin{align*}
    x=\max\left\{\frac{\ell_{\rho_1}(\alpha\beta)}{2},\frac{\ell_{\rho_1}(\alpha\beta^{-1})}{2}\right\}, \quad y=\min\left\{\frac{\ell_{\rho_1}(\alpha\beta)}{2},\frac{\ell_{\rho_1}(\alpha\beta^{-1})}{2}\right\}  
\end{align*} 
and define a function $f$:
\begin{equation*}
    f(\lambda)= 2\cosh^2\left(\lambda a\right)+\cosh\left(\lambda y\right)-\cosh\left(\lambda x\right) .
\end{equation*}
From the equality (\ref{eq:embedineqneg}) we have that $f(1)=0$. Differentiating with respect to $\lambda$ and using that $x\geq y$ gives:
\begin{equation*}
    \begin{split}
        f'(\lambda)&= 4a\cosh(\lambda a)\sinh(\lambda a)+y\sinh(\lambda y) -x\sinh(\lambda x)\\
        &\leq 2a\sinh(2\lambda a)+x\sinh(\lambda y) -x\sinh(\lambda x)\\
        &< 2a\sinh(2\lambda a)+x(\cosh(\lambda y) -\cosh(\lambda x)+1)\\
        &= 2a\sinh(2\lambda a)+x(f(\lambda)-2\cosh^2(\lambda a)+1)\\
        &= 2a\sinh(2\lambda a)+x(f(\lambda)-\cosh(2\lambda a))\\
        &<(2a-x)\cosh(2\lambda a)+x f(\lambda)
    \end{split}
\end{equation*}
Therefore as $x>\frac{\ell_{\rho_1}(\alpha)+\ell_{\rho_1}(\beta)}{2}=2a$ we have that $f'(1)<0$ and thus $f(\lambda)<0$ for all $\lambda>1$.
Therefore $\lambda_0=1$ is the only possible scalar.
\end{proof}
Lemma \ref{lem:embedtoproj} implies that we can identify $\reps$ with its image in $\projc$.
\medskip

There is a natural action of the automorphism group $\aut$ on the space $\reps$ given by composition $\phi\cdot \rho\mapsto\rho\circ\phi^{-1}$ for any $\phi\in\aut$ and $\rho\in\reps$. As elements of $\reps$ are defined up to conjugation then the inner automorphisms of $\free$ act trivially and thus we have a natural action of $\out$. This action is discrete. 

Once we recall that $\reps$ is the space of markings $\free\rightarrow\Lambda$ for the set of all subgroups $\Lambda\in\subgrps$ we see that the $\out$ action on $\reps$ amounts to a change of the markings $\free\rightarrow\Lambda$, or equivalently, a change in the choice of generators for each $\Lambda$.
This is analogous to the mapping class group action on Teichm\"{u}ller space. 
Thus we endow the set of unmarked subgroups $\subgrps$ with a topology by viewing it as the quotient space of the action of $\out$ on $\reps$.
\medskip

We need to be able to describe what happens when a sequence of subgroups in $\surfsubgrp$ diverge inside $\subgrps$, so our aim is to add a boundary to the space $\subgrps$ to which divergent sequences of subgroups accumulate. In order to do this we will first explain how Culler-Vogtmann Outer space can be used as a boundary for the space of representations $\reps$.

The space $\reps$ of discrete faithful representations $\free\rightarrow\psl$ has a compactification $\overline{\reps}$ such that each element of the boundary $\overline{\reps}-\reps$ corresponds to the length function of a small minimal action of $\free$ on an $\R$-tree via isometries where two actions are equivalent if there exists an equivariant isometry between them. The length function of an isometric action of $\free$ on a metric tree $T$ is given by:
\begin{equation*}
    \ell_T:\free \rightarrow\R_{\geq0}, \quad
    \alpha \mapsto\inf_{x\in T}\{d_{T}(x,\alpha\cdot x)\} 
\end{equation*}
and these lengths are invariant under conjugation in $\free$. For an element of the boundary $\overline{\reps}-\reps$ the metric of the corresponding $\R$-tree is well-defined up to scaling. For details see the work of Culler and Shalen \cite{MS84}, Culler and Morgan \cite{CM87} and Bestvina \cite{Best88}.

However, the action of $\out$ on the compactification $\overline{\reps}$ is `not nice': it is not properly discontinuous and can have large stabilisers. Thus, instead of considering the full boundary $\overline{\reps}-\reps$, we will focus on a subset of the boundary with a `nice' $\out$ action. Specifically we restrict to the subspace corresponding to actions on metric simplicial trees, that is, to Culler-Vogtmann Outer space.

Culler and Vogtmann in \cite{CV86} introduced (projectivised) Outer space, denoted $\cv$, as the space of free minimal isometric actions of $\free$ on metric simplicial trees (such that quotient graph has volume one). Again two such actions are considered equivalent if there is a equivariant isometry between the trees. There is a properly discontinuous action of $\out$ on $\cv$ which changes the isometric action of $\free$ by composing with an automorphism but preserves the metric on the tree. See \cite{Vogt15} for more details on Outer space.

As was the case with representations, points in Outer space are uniquely identified by their length functions \cite{CV86}. 
Thus projectivised length functions give an embedding of Outer space into projective space, $\cv\hookrightarrow\projc$, and we identify $\cv$ with its image. 

We will denote by $\hreps=\reps\cup \cv$ the union of the space of representations with Outer space. The space $\hreps$ is not a compactification of $\reps$, it is however locally compact:
\begin{lemma}\label{lem:localcompact}
    The space $\hreps:= \reps\cup\cv$ is locally compact.
\end{lemma}
\begin{proof}
    Let $x\in\hreps$ and suppose that $\hreps$ is not locally compact at this point, meaning that all neighbourhoods $U\ni x$ contain a sequence without a convergent subsequence in $\hreps$. 

    Take a sequence of open neighbourhoods $U_i\subset\hreps$ of $x$ with $U_{i+1}\subset U_i$ such that $\bigcap_i U_i=\{x\}$. By assumption, and as $\overline{\reps}$ is compact, each $U_i$ contains a sequence with a subsequence convergent to a point $p_i\in\overline{\reps}-\hreps$ not in the space $\hreps$. Each limit point $p_i$ is the length function for very small, minimal, $F_k$-equivariant isometric action on an $\R$-tree \cite{MS84} and the sequence of limit points $(p_i)$ converges to $x$.
    
    Suppose $x\in\cv$, then as Outer space is open in the boundary $\overline{\reps}-\reps$ \cite{CV86} there must be some $i$ such that $p_i\in\cv$, but this contradicts our construction of $(p_i)$. 
    
    Similarly, as $\reps$ is open in $\overline{\reps}$ \cite{Best88}, then if $x\in\reps$ the sequence $(p_i)$ must at some point land in $\reps$. 
    This is again a contradiction, thus $\hreps$ is locally compact.
\end{proof}
Outer space $\cv$ can also be described as the space of marked (volume one) metric graphs of rank $k$ \cite{CV86}. By a marking on a graph we refer to a homotopy equivalence $f: R_k\rightarrow X$ between the rose $R_k$ with $k$ petals and a finite graph $X$, where we have identified the free group $\free$ with the rose $R_k$. For each marking $f: R_k\rightarrow X$ we can represent the possible (volume one) metric structures with an $(E-1)$-dimensional open simplex where $E$ is the number of edges of $X$. In this way we can view outer space as a union of open simplices corresponding to marked graphs where the topology arising from the simplicial structure coincides with the projective length function topology \cite{Paul89}. 

The action of $\out$ on $\cv$ acts by changing the marking of the graph while preserving the metric, notably it is a simplicial action. Note that the stabiliser of any point in Outer space is finite as it corresponds exactly to the group of isometries of the underlying graph. The $\out$-action has finite quotient which we call the moduli space of (volume one) rank $k$ metric graphs, denoted $\gmoduli=\cv/\out$. For every topological type of graph with rank $k$, moduli space contains a quotiented simplex corresponding to the volume one metrics on this graph, parametrised by edge-lengths, up to automorphisms of the graph.

We define the space $\hsubgrps$ as the union of our space of discrete subgroups of $\psl$ and the moduli space of metric graphs:
\begin{equation*}
    \hsubgrps:=\subgrps\cup\gmoduli 
\end{equation*}
equipped with the quotient topology coming from the action of $\out$ on the space $\hreps=\reps\cup\cv$ endowed with the length function topology as above.

Now we will consider $\hsubgrps=\subgrps\cup\gmoduli$ to be the ambient space in which the set $\surfsubgrp$ of conjugacy classes of free subgroups of $\fund$ sits as a cloud of points. In Section $\ref{sec:measures}$ of this paper we will define a family of measures on $\hsubgrps$ each supported on some finite subset of $\surfsubgrp$. We will then prove Theorem \ref{thm:limitmeasure} which states that this family of measures converges to a probability measure supported on $\gmoduli$ and that this limit measure is of the Lebesgue class.

\section{Counting conjugacy classes of subgroups}\label{sec:graphs}
As in the introduction we let $\surfsubgrp(L)$ denote, for any $L>0$, the finite subset of $\surfsubgrp$ containing conjugacy classes of subgroups of length at most $L$. The length here arises from the geodesic length of certain graphs immersed in the surface and we begin this section by introducing these graphs and their correspondence with the subgroups. Our goal is then to prove Theorem \ref{thm:countingsubgrps} which gives the asymptotic growth of the cardinality of $\surfsubgrp(L)$ as $L\to\infty$.
\medskip

A \textbf{carrier graph} for a subgroup $\Delta\in\surfsubgrp$ is a topological graph $X$ along with a $\pi_1$-injective continuous map $\phi:X\rightarrow\Sigma$ with $\phi_*(\gfund)=\Delta$.
This is an analogue of carrier graphs for hyperbolic 3-manifolds first introduced by White \cite{White02}. 
Our carrier graphs are immersions of a graph into the surface, we stress that they are not necessarily embeddings.
\begin{figure}[ht]
    \centering
    \includegraphics[width=0.8\linewidth]{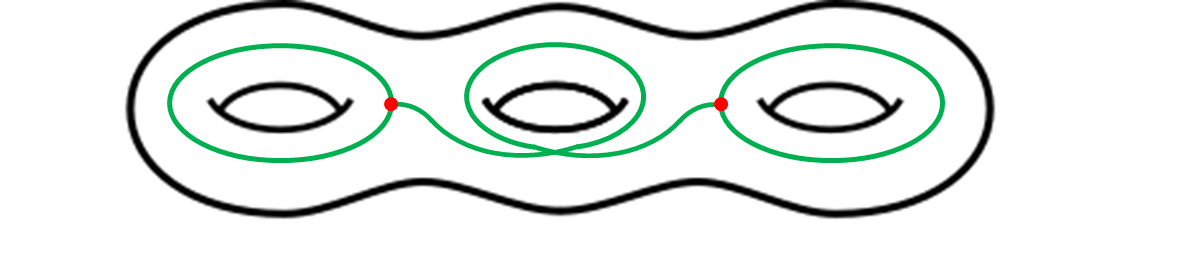}
    \caption{The image of a rank $2$ carrier graph on a genus 3 surface. The underlying graph is the dumbbell}
    \label{fig:carrieronsurf}
\end{figure}

Note that we specify graphs to be finite, connected and with all vertices degree greater than or equal to 3. 
Thus, for every $k$, we only have a finite number of homeomorphism classes of graphs with fundamental group of rank $k$.
The length of a carrier graph $(X,\phi)$ is defined to be the sum of the lengths of the images of its edges:
\begin{equation*}
    \ell_{\Sigma}(X,\phi)= \sum_{e\in\mathbf{edge}(X)}\ell_\Sigma(\phi(e)) \, .
\end{equation*}
Every subgroup $\Delta\in\surfsubgrp$ is carried by an infinite number of graphs and we define the length of $\Delta$ to be the infimum of the lengths of its carrier graphs
\begin{equation}\label{eq:defsublength}
    \ell(\Delta):=\inf\left\{\ell_\Sigma(X,\phi)\mid (X,\phi) \text{ is a carrier graph of } \Delta\right\}  .
\end{equation}
This length function is such that for every $L$ there are a finite number of elements of $\surfsubgrp$ with length less than or equal to $L$.
 
A carrier graph $(X,\phi)$ of a subgroup $\Delta\in\surfsubgrp$ is said to be \textbf{minimal} if $\ell_\Sigma(X,\phi)=\ell(\Delta)$.
White shows that any closed hyperbolic 3-manifold has a carrier graph for its fundamental group which achieves minimal length \cite[Lemma 2.2]{White02}. For our closed hyperbolic surface $\Sigma$ we can use a similar argument to show that every subgroup $\Delta\in\surfsubgrp$ has some carrier graph of minimal length. 

\begin{lemma}\label{lem:mingraphexists}
    There exists a minimal carrier graph for every subgroup $\Delta\in\surfsubgrp$.
\end{lemma}
\begin{proof}
    Pick some graph $X$ with $\pi_1(X)\cong\free$. Let $\ell_X$ be the infimum of the lengths of maps $X\rightarrow\Sigma$ carrying $\Delta$:
    \begin{equation*}
        \ell_{X}:=\inf\{\ell_{\Sigma}(X,\phi)\mid\phi:X\rightarrow\Sigma \text{ carries }\Delta\} \,.
    \end{equation*} 
    We can construct a family of maps $\mathcal{F}_{X}=\{\phi_i:X\rightarrow\Sigma\mid (X,\phi_i)\text{ carries }\Delta\}_{i=1}^\infty$ such that each $(X,\phi_i)$ is a carrier graph for $\Delta$ and the lengths $\ell_\Sigma(X,\phi_i)\to \ell_{X}$ as $i\to\infty$. As $\Sigma$ is a compact space, the Arzela-Ascoli theorem shows us that the closure of $\mathcal{F}_{X}$ is compact in the set of continuous maps $C(X,\Sigma)$. Thus there exists a subsequence $\phi_j$ of $\mathcal{F}_{X}$ which converges uniformly to a continuous map $\psi:X\rightarrow\Sigma$ with $\ell_{\Sigma}(X,\psi)=\ell_{X}$ and such that $(X,\psi)$ carries $\Delta$. 
    
    As there are finitely many topological types of graphs with fundamental group $\free$, we can repeat the process above for each of these graphs and take the carrier graph of minimal length. 
\end{proof}

Note that although all subgroups $\Delta\in\surfsubgrp$ have a minimal carrier graph it may not be unique. For example, if $\Delta$ has a minimal carrier graph $(X,\phi)$ such that the underlying graph $X$ has a non-trivial automorphism $A:X\rightarrow X$ then the carrier graph $\phi\circ A:X\rightarrow\Sigma$ is also minimal. 
There are also subgroups which have two minimal carriers with different underlying topological graphs, however we will see later in this section that this situation rarely occurs.

In the setting of closed 3-manifolds, White showed that if a graph carrying the fundamental group of the manifold has minimal length then it is necessarily geodesic and trivalent, with half edges spaced evenly around vertices \cite{White02}. Essentially the same argument can be applied to the 2-dimensional case and for graphs carrying subgroups of fundamental groups, we call such a graph critical:
\begin{definition}\label{def:critmap}
    A \textbf{critical graph map} for a surface $\Sigma$ is a connected trivalent graph $X$ along with a continuous map $\phi:X\rightarrow\Sigma$ such that edges are sent to non-degenerate geodesic arcs which meet at vertices with angles $\frac{2\pi}{3}$.
\end{definition}
The key idea of White is that any non-critical map can be made shorter by small local deformations and therefore cannot have minimal length.
For instance, take a carrier graph $(X,\phi)$ for $\Delta\in\surfsubgrp$ and suppose that there exists an edge $e\in\mathbf{edge}(X)$ such that $\phi(e)$ is not a geodesic segment. We can homotop $\phi(e)$ fixing the vertices to get a new shorter carrier graph of $\Delta$ where $e$ is sent to a geodesic arc. Thus a minimal carrier graph must map edges to geodesic arc segments.

If a carrier graph is such that two incident edges of a degree 3 vertex meet at an angle less than $\frac{2\pi}{3}$ then we are able to deform the carrier graph by `pulling' the vertex in such a way which reduces the length, see Figure \ref{fig:optimalangle}.
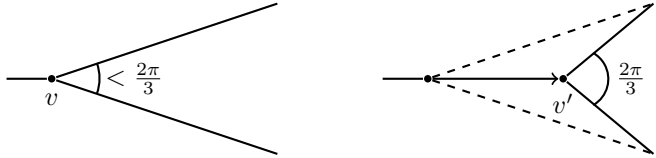
\begin{figure}[ht]
\centering
\begin{tikzpicture}[scale=2, thick]
\node[circle, fill=black, inner sep=1pt, label=below:$v$] (v) at (-0.5,0.5) {};
\draw (v) -- (1,0) node[right] {};
\draw (v) -- (1,1) node[above right] {};
\draw (v) -- (-0.8,0.5) node[above left] {};
\draw (-0.2,0.4) arc[start angle=-20, end angle=20, radius=0.3];
\node at (0.05,0.5) {$<\frac{2\pi}{3}$};

\node[circle, fill=black, inner sep=1pt] (vn) at (2,0.5) {};
\draw[dashed] (vn) -- (3.5,0) node[right] {};
\draw[dashed] (vn) -- (3.5,1) node[above right] {};
\draw (vn) -- (1.7,0.5) node[above left] {};
\node[circle, fill=black, inner sep=1pt, label=below:$v'$] (vp) at (2.9,0.5) {};
\draw (vp) -- (3.5,0);
\draw (vp) -- (3.5,1);
\draw[->, thick] (vn) -- (vp);
\draw (3.1,0.325) arc[start angle=-60, end angle=60, radius=0.2];
\node at (3.35,0.5) {$\frac{2\pi}{3}$};
\end{tikzpicture}
\caption{Moving a trivalent vertex so that all angles are equal to $\frac{2\pi}{3}$ reduces the overall length}
\label{fig:optimalangle}
\end{figure}

Similarly, if a vertex has degree greater than 3 then two incident edges must meet at an angle less than $\frac{2\pi}{3}$. We can then construct a new shorter carrier graph by `splitting' this vertex into two new vertices of smaller degree, see Figure \ref{fig:trivalenceofvert}. 
\begin{figure}[ht]
\centering
\begin{tikzpicture}[scale=2, thick]
\node[circle, fill=black, inner sep=1pt, label=below:$v$] (v) at (-0.6,0.6) {};
\draw (v) -- (1,0) node[right] {};
\draw (v) -- (1,1) node[above right] {};
\draw (v) -- (-0.8,1.2) node[above left] {};
\draw (v) -- (-1.2,0.2) node[below left] {};
\draw (-0.3,0.48) arc[start angle=-20, end angle=20, radius=0.3];
\node at (-0.05,0.55) {$<\frac{2\pi}{3}$};

\node[circle, fill=black, inner sep=1pt, label=below:$v$] (vn) at (2.4,0.6) {};
\draw[dashed] (vn) -- (4,0) node[right] {};
\draw[dashed] (vn) -- (4,1) node[above right] {};
\draw (vn) -- (2.2,1.2) node[above left] {};
\draw (vn) -- (1.8,0.2) node[below left] {};
\node[circle, fill=black, inner sep=1pt, label=below:$v'$] (vp) at (3.4,0.5) {};
\draw (vp) -- (4,0);
\draw (vp) -- (4,1);
\draw[->, thick] (vn) -- (vp);
\draw (3.6,0.325) arc[start angle=-60, end angle=60, radius=0.2];
\node at (3.85,0.5) {$\frac{2\pi}{3}$};
\end{tikzpicture}
\caption{Splitting a vertex with degree $>3$ reduces the overall length}
\label{fig:trivalenceofvert}
\end{figure}
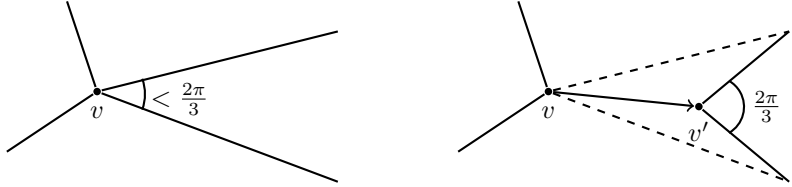
We can keep performing this action and reducing length until all vertices are degree 3.

By definition a minimal carrier graph cannot be deformed to reduce its length. Therefore it follows that any minimal length carrier graph must be trivalent with non-degenerate geodesic edges which meet at angles $\frac{2\pi}{3}$. For future reference we state this as a lemma:
\begin{lemma}\label{lem:minarecrit}
    Minimal carrier graphs are always critical. \qed
\end{lemma}

A result of Erlandsson and Souto \cite[Theorem~1.3]{ES23-commutator} utilises a variation of Delsarte's lattice point counting theorem \cite{Del42} to give the asymptotic growth of the number of critical maps of a fixed graph:
\begin{prop}\label{prop:countcrit}
    Let $X$ be a connected trivalent graph of some fixed topological type. Let $\critx(L)$ be the set of critical graph maps $\phi:X\rightarrow\Sigma$ with length less than or equal to $L$. Then we have
    \begin{equation*}
        |\critx(L)|\sim \left(\frac{2}{3}\right)^{3\chi(X)}\cdot\frac{\text{vol}(T^1\Sigma)^{\chi(X)}}{(-3\chi(X)-1)!}\cdot L^{-3\chi(X)-1}\cdot e^L
    \end{equation*}
    as $L\to\infty$, where $\chi(X)$ is the Euler-characteristic of the graph $X$.
\end{prop}

If the induced homomorphism of a critical graph map is faithful then we see that the map is also a carrier graph. However, this is not always the case.
For example, for any $n\geq7$ there exists a tiling of the hyperbolic plane by regular $n$-sided polygons with interior angles $\frac{2\pi}{3}$, see Figure \ref{fig:hepttiling}. The $1$-skeleton of such a tiling is an infinite trivalent graph with geodesic edges meeting at vertices with angle $\frac{2\pi}{3}$. 

\begin{figure}[ht]
    \centering
    \includegraphics[width=0.5\linewidth]{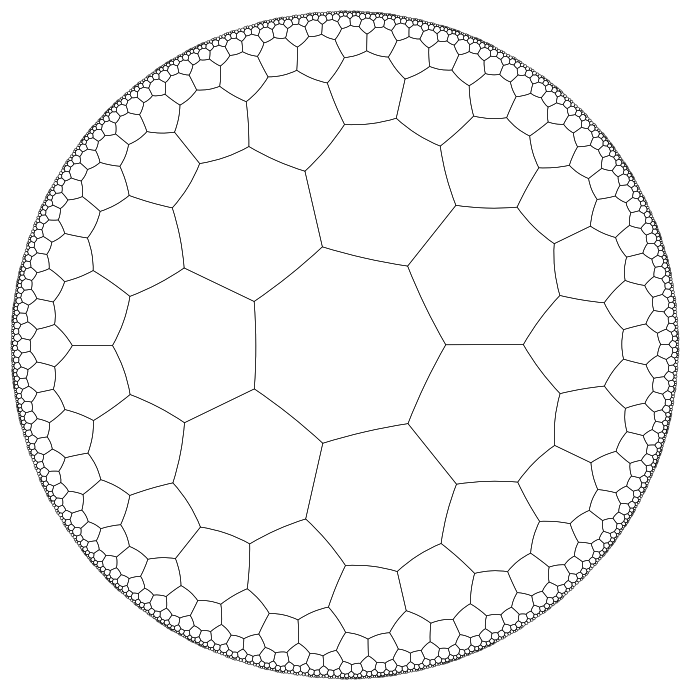}
    \caption{Tiling of the hyperbolic plane by heptagons with angles $\frac{2\pi}{3}$, made with `hypertiling' \cite{hyptiling}}
    \label{fig:hepttiling}
\end{figure}
If $\Sigma$ is a closed surface obtained as a quotient of $\hyp$ by a torsion-free finite index subgroup of the symmetry group of the tiling, then the projection $p:\hyp\to\Sigma$ maps the $1$-skeleton to a finite trivalent geodesic graph on $\Sigma$. This inclusion is a critical graph map, but the boundary cycle of any polygon tile is null-homotopic in $\Sigma$, so the induced homomorphism is not faithful.

We want to restrict our attention to the critical graph maps which are also carrier graphs. Fortunately, non-faithful critical graphs have certain strict geometric constraints, specifically there are limits on their minimum edge lengths. For example, in the above hyperbolic tiling construction some elementary hyperbolic geometry shows that a regular hyperbolic polygon with interior angles $\frac{2\pi}{3}$ must have edge lengths less than $\log(3)\approx 1.0986$.

We make this precise by introducing the notion of long edge lengths. We say a graph map $(X,\phi)$ is \textbf{$\ell$-long} if for all edges $e\in\mathbf{edge}(X)$ we have $\ell_\Sigma(\phi(e))>\ell$.

\begin{lemma}\label{lem:longcritcarry}
    For a hyperbolic surface $\Sigma$ there exists some $\ell_0>0$ such that any $\ell_0$-long critical graph map $(X,\phi)$ on $\Sigma$ with Euler-characteristic $\chi(X)=1-k$ is a carrier graph for some subgroup in $\surfsubgrp$.
\end{lemma}
\begin{proof}
    Suppose, we have an $\ell_0$-long critical graph $(X,\phi)$ which is not a carrier graph, meaning the induced homomorphism $\phi_*:\gfund\rightarrow\fund$ is not injective. 

    Pick a basepoint $x_0\in X$ and let $\phi(x_0)\in \Sigma$ be a basepoint for the surface. Choosing lifts $\tilde{x}_0$ and $\widetilde{\phi(x_0)}$ to the universal covers $\tilde{X}$ and $\hyp$ uniquely determines a lift $\tilde{\phi}:\tilde{X}\rightarrow\hyp$. As $(X,\phi)$ is critical then edges of $\tilde{X}$ must be sent by $\tilde{\phi}$ to geodesic segments in $\hyp$ which meet at angles $\frac{2\pi}{3}$. The non-injectivity condition implies that there exists a non-trivial loop in the image $\tilde{\phi}(\tilde{X})$.  

    Suppose $x,y\in \mathbf{vert}(\tilde{X})$ are distinct vertices which under $\tilde{\phi}$ are sent to the same point in $\hyp$, $\tilde{\phi}(x)=\tilde{\phi}(y)$. In the tree $\tilde{X}$ there is a unique path joining $x$ and $y$ and we say $d_{\tilde{X}}(x,y)$ is the length of this path where each edge $e\in\mathbf{edge}(\tilde{X})$ has length $\ell_{\hyp}(\tilde{\phi}(e))$. 
    Let $n\geq 1$ be the number of edges between $x$ and $y$, then we have:
    \begin{equation*}
        d_{\tilde{X}}(x,y)>n\ell_0 \, .
    \end{equation*}

    In the hyperbolic plane it is known that for some large enough $L>0$, a path comprised of geodesic segments each longer than $L$ which meet angles $\frac{2\pi}{3}$ is a quasi-geodesic \cite{BH99}. Thus, assuming that $\ell_0\geq L$, there exist constants $\lambda\in(0,1]$ and $C\geq 0$ such that:
    \begin{equation*}
        \begin{split}
            d_{\hyp}(\tilde{\phi}(x),\tilde{\phi}(y)) &\geq\lambda d_{\tilde{X}}(x,y)- C \\
            & >\lambda n\ell_0 - C 
        \end{split}
    \end{equation*}
    and so for large enough $\ell_0$ we must have $d_{\hyp}(\tilde{\phi}(x),\tilde{\phi}(y))>0$ which contradicts the fact that $\tilde{\phi}(x)=\tilde{\phi}(y)$. Therefore, for such an $\ell_0$, there are no non-trivial loops in the image $\tilde{\phi}(\tilde{X})$ and so the induced homomorphism $\phi_*$ is faithful.
\end{proof}

Given a trivalent topological graph $X$ and any $\ell>0$ a result of Erlandsson and Souto \cite[Lemma~4.3]{ES23-commutator} states the the proportion of $\ell$-long elements in the set $\critx(L)$ tends to 1 as $L\to\infty$. The argument naturally extends to the set $\critgraph=\sum_{X\in\typ}\critx$ of all rank $k$ critical graph maps where $\typ$ denotes the set of homeomorphism classes of connected trivalent graphs of rank $k$:
\begin{lemma}\label{lem:shortnegl}
    For any $\ell$, we have that
    \begin{equation*}
        \liminf_{L\to\infty}\frac{|\{(X,\phi)\in\critgraph(L):(X,\phi) \text{ is } \ell\text{-long}\}|}{|\critgraph(L)|}=1 \, .
    \end{equation*}
\end{lemma}

Therefore generically most critical graphs have long edges, and thus as $L\to\infty$ the number of critical graph maps which are not carrier graphs becomes negligible.

Our aim now is to show that generically there is a one-to-one relationship between conjugacy classes of subgroups and automorphism classes of critical graph maps. 

\begin{prop}\label{prop:uniquecrit}
    There exists $\ell_0=\ell_0(\Sigma,k)>0$ such if a element $\Delta\in\surfsubgrp$ is carried by an $\ell_0$-long critical graph $(X, \phi)$ then all other $\ell_0$-long critical graphs carrying $\Delta$ differ only by composition with a homeomorphism of the graph $X$. 
\end{prop}
\begin{proof}
Fix some $\Delta\in\surfsubgrp$ and let $(X,\phi)$ and $(Y,\psi)$ be $\ell_0$-long critical graphs which carry $\Delta$ where $\ell_0$ is large enough to satisfy conditions given later in this proof.

Suppose, without loss of generality, that $\ell_{\Sigma}(X,\phi)\geq\ell_\Sigma(Y,\psi)$. For the graph $X$ pick a basepoint $x_0\in X$ and let $\phi(x_0)\in \Sigma$ be a basepoint for $\Sigma$. By slightly abusing notation, we allow $\Delta$ to denote the subgroup $\phi_*(\pi_1(X,x_0))\subset \pi_1(\Sigma, \phi(x_0))$ as a representative in its conjugacy class. This gives us a cover $p:\hyp/\Delta\rightarrow\Sigma$ and lifts of $\phi$ and $\psi$ to maps $\hat{\phi}:X\rightarrow \hyp/\Delta$ and $\hat{\psi}:Y\rightarrow \hyp/\Delta$. We then have an isomorphism $\hat{\psi}_*^{-1}\circ \hat{\phi}_*: \pi_1(X)\rightarrow\pi_1(Y)$ and thus there exists a (free) homotopy equivalence $\sigma:X\rightarrow Y$ with $\hat{\phi}\simeq\hat{\psi}\circ\sigma$ which descends to a homotopy equivalence $\phi\simeq\psi\circ\sigma$.

Using a result of Erlandsson and Souto \cite[Proposition 2.3]{ES23-commutator} then gives us that for $\ell_0$ large enough there exists a homeomorphism $F:Y\rightarrow X$ mapping each edge with
constant speed such that $F\circ \sigma$ is homotopic to the identity. Therefore $\psi$ and $\phi\circ F$ are homotopic critical graph maps from $Y$ to $\Sigma$. The argument preceding Lemma \ref{lem:minarecrit} shows that any small deformation of a critical graph must increase its length, thus both $\psi$ and $\phi\circ F$ are local minima of the length function. However, convexity of the distance function $d_{\hyp}(\cdot,\cdot)$ implies that the length function $\ell_\Sigma$ is convex on the space of geodesic carrier graphs $Y\rightarrow\Sigma$ homotopic to $\psi$. Thus there exists a unique minimum for the length function on this space and therefore we have $\psi=\phi\circ F$.
\end{proof}

\begin{coroll}\label{coroll:nummincar}
There exists $\ell_0=\ell_0(\Sigma,k)>0$ such if a element $\Delta\in\surfsubgrp$ has an $\ell_0$-long minimal carrier graph $(X,\phi)$ then $\Delta$ has exactly $\left|\gaut\right|$ minimal carrier graphs, where $\gaut$ is the set of automorphisms of the graph $X$.
\end{coroll}
\begin{proof}
    Take the $\ell_0$ given by Proposition \ref{prop:uniquecrit} and let $(X,\phi)$ be an $\ell_0$-long minimal carrier graph of $\Delta\in\surfsubgrp$. Suppose $(Y, \psi)$ is also a minimal carrier graph. Lemma \ref{lem:minarecrit} tells us that both $(X,\phi)$ and $(Y, \psi)$ are critical. 

    In the proof of the above proposition, the result used to find the homeomorphism $F:Y\rightarrow X$ only requires that the longer of the two critical graphs is $\ell_0$-long. As both $(X,\phi)$ and $(Y, \psi)$ are minimal we have $\ell(X,\phi)=\ell(Y,\psi)$, thus the proof of Proposition \ref{prop:uniquecrit} still holds and we have that $\psi =\phi\circ F$ for some homeomorphism $F:Y\rightarrow X$ mapping edges with constant speed. Viewing $X$ as a homeomorphism class of graphs means that $Y=X$ and $F$ can be viewed as a graph automorphism $F:X\rightarrow X$.
    
    Moreover, as discussed above, any automorphism $F\in\gaut$ gives a distinct minimal carrier graph $(X,\phi\circ F)$ of $\Delta$, thus the number of minimal carriers is $\left|\gaut\right|$.
\end{proof}

We now have the results we need to prove our theorem:
\begin{thm:countsubgroup}
    For a closed, orientable hyperbolic surface $\Sigma$ of genus $g$ and some $k\geq 2$, the cardinality of the set $\surfsubgrp(L)$ has asymptotic growth:
    \begin{equation*}
        \left|\surfsubgrp(L)\right|\sim \sum_{X\in\typ}\frac{1}{\left|\gaut\right|} \cdot\left(\frac{4}{3}\right)^{3-3k}\cdot\frac{(\pi^2(g-1))^{1-k}}{(3k-4)!}\cdot L^{3k-4}\cdot e^L
    \end{equation*}
    as $L\to\infty$.
\end{thm:countsubgroup}
\begin{proof}
    By definition any subgroup $\Delta\in\surfsubgrp(L)$ is carried by a minimal carrier graph of length less than or equal to $L$. Lemma \ref{lem:minarecrit} tell us that this minimal carrier graph is critical, thus $\Delta$ has at least one carrier graph in the set $\critgraph(L)$ of rank $k$ critical graph maps of length less than or equal to $L$. 

    Combining Lemma \ref{lem:longcritcarry} and Proposition \ref{prop:uniquecrit} we get that there exists $\ell_0$ such that any $\ell_0$-long critical graph $(X,\phi)\in\critgraph(L)$ is a carrier graph of a subgroup $\Delta\in\surfsubgrp(L)$ and that the number of $\ell_0$-long critical carriers of $\Delta$ is exactly $|\gaut|$. Then from Lemma \ref{lem:shortnegl} we get that the number of subgroups $\Delta\in\surfsubgrp(L)$ carried by critical maps which are not $\ell_0$-long becomes negligible as $L\to\infty$. Thus it follows that:  
    \begin{align*}
        |\surfsubgrp(L)|&\sim\sum_{X\in\typ}\frac{|\critx(L)|}{\left|\gaut\right|} \\
        &\sim \sum_{X\in\typ}\frac{1}{\left|\gaut\right|}\cdot \left(\frac{2}{3}\right)^{3\chi(X)}\cdot\frac{\text{vol}(T^1\Sigma)^{\chi(X)}}{(-3\chi(X)-1)!}\cdot L^{-3\chi(X)-1}\cdot e^L
    \end{align*}
    as $L\to\infty$, where we have used Proposition \ref{prop:countcrit} to give asymptotic growth of the number of critical graph maps. Once we note that any rank $k$, trivalent graph $X$ has Euler-characteristic $\chi(X)=1-k$ and any genus $g$ surface $\Sigma$ has $\text{vol}(T^1\Sigma)=2\pi\cdot\text{vol}(\Sigma)=8\pi^2(g-1)$ the result follows directly.
\end{proof}

\section{Convergence of probability measures}\label{sec:measures}
As mentioned in the introduction we can define, for every length $L>0$, a probability measure:
\begin{equation*}
    \measl := \frac{1}{|\surfsubgrp(L)|}\sum_{\Delta\in\surfsubgrp(L)}\delta_{\Delta}
\end{equation*}
on the space $\hsubgrps$ as defined in Section \ref{sec:spaces} where $\delta_{\Delta}$ is the Dirac measure. The measure is supported on the elements of $\surfsubgrp(L)$ and is simply a counting measure normalised to be a probability measure. 

Recall that the space $\hsubgrps$ is the quotient of $\hreps$, a locally compact Hausdorff space (Lemma \ref{lem:localcompact}), by a properly discontinuous action of $\out$ and therefore $\hsubgrps$ is itself locally compact and Hausdorff. Let $\mathbf{M}(\hsubgrps)$ be the set of Radon measures on the space $\hsubgrps$. We endow $\mathbf{M}(\hsubgrps)$ with the weak-*-topology which is the weakest topology such that the map
\begin{equation*}
    \mathbf{M}(\hsubgrps) \rightarrow\R  , \quad \mu \mapsto \int f d\mu
\end{equation*}
is continuous for all compactly supported continuous functions $f\in C_c(\hsubgrps)$.

We can now restate our main theorem:
\begin{thm:limitmeasure}
    The measures $\measl$ on $\hsubgrps$ converge with respect to the weak-*-topology to a limit measure
    \begin{equation*}
        \meas= \lim_{L\to\infty}\measl
    \end{equation*}
    supported on the moduli space of (volume one) metric graphs $\gmoduli$. This limit is a probability measure of full support in the Lebesgue class which is moreover independent of the surface $\Sigma$. 
\end{thm:limitmeasure}

As we mentioned previously, the limiting measure $\meas$ is explicit, see (\ref{eq:limitmeas}) below. To be able to describe it recall that the moduli space $\gmoduli$ of (volume one) rank $k$ metric graphs is the finite quotient of the action of $\out$ on Outer space $\cv$, the space of (volume one) marked rank $k$ metric graphs. Outer space can be described as a union of open simplices with face identifications, where each simplex corresponds to edge lengths of a marked graph and faces are identified when edges in a graph collapse. The action of $\out$ on $\cv$ is properly discontinuous and sends simplices to simplices where the stabiliser of a point in Outer space is isomorphic to the group of isometries of the metric graph \cite{Vogt15}. 

Moduli space $\gmoduli$ is the quotient space of the action of $\out$ on $\cv$. The simplicial structure of $\cv$ descends to a finite orbifold simplicial structure, where each topological graph has a single `folded' simplex corresponding to its possible volume one metrics. Explicitly, for every rank $k$ graph $X$ we have a corresponding space 
\begin{equation*}
    S_X:=\{\vec{x}\in \R^{\mathbf{edges}(X)}_+ : ||\vec{x}||=1\}/\sim
\end{equation*}
sitting inside $\gmoduli$ where $||\vec{x}||=x_1+\dots+x_E$ is the $L^1$-norm and $\vec{x}\sim\vec{y}$ if there exists a graph automorphism $A\in\gaut$ which acts on edges such that $A\vec{x}=\vec{y}$. Note that some graph automorphisms may act trivially on the edge lengths of a graph and we denote this subgroup: 
\begin{equation}\label{def:triv}
    \operatorname{Triv}(X):=\left\{A\in\gaut \mid A\vec{x}=\vec{x} \text{ for all }\vec{x}\in \R^{\mathbf{edges}(X)}_+\right\}  . 
\end{equation}

The top-dimensional simplices correspond to trivalent graphs and are $(3k-4)$-dimensional. For each trivalent graph $X\in\typ$ we will define a Lebesgue-type measure $\sigma_X$ on $S_X$ as the limit of a family of integer lattice point counting measures:
\begin{equation}\label{def:sigma-meas}
    \sigma_X = \lim_{N\to\infty}\frac{1}{|\{\vec{n}\in\N^{3k-3} : ||\vec{n}||=N\}|}\sum_{\substack{[\vec{x}]\in S_X \\ N\vec{x}\in\N^{3k-3}}} \delta_{[\vec{x}]} 
\end{equation} 
where $3k-3$ is the number of edges of a trivalent graph of rank $k$.
Note that for any $N\in\N$:
\begin{equation*}
    |\{[\vec{x}]\in S_X : N\vec{x}\in\N^{3k-3}\}|=\frac{|\operatorname{Triv}(X)|}{|\gaut|}|\{\vec{n}\in\N^{3k-3} : ||\vec{n}||=N\}|
\end{equation*}
and thus $\operatorname{vol}_{\sigma_X}(S_X)=\frac{|\operatorname{Triv}(X)|}{|\gaut|}$.

We will see that the limit measure in Theorem \ref{thm:limitmeasure} is:
\begin{equation}\label{eq:limitmeas}
    \meas = \frac{1}{\sum_{X\in\typ}\frac{1}{|\gaut|}} \cdot \sum_{X\in\typ}\frac{1}{|\gtriv|} \sigma_{X} \, .
\end{equation}

Before we give the proof of Theorem \ref{thm:limitmeasure} we first must introduce a new family of measures living on $\gmoduli$, the moduli space of metric graphs.

Recall from Corollary \ref{coroll:nummincar} that there exists $\ell_0>0$ such that if an element $\Delta\in\surfsubgrp$ is carried by an $\ell_0$-long minimal carrier graph $(X,\phi)$ then its only other minimal carrier graphs are obtained by graph automorphisms of $(X,\phi)$. Specifically, such a $\Delta$ has a unique underlying metric graph associated to it via its minimal carrier graphs. 
Fix this $\ell_0$ and denote by $\surfsubgrp^{\ell_0}$ the subset of $\surfsubgrp$ consisting of those subgroups with an $\ell_0$-long minimal carrier graph. We can define a projection $p:\surfsubgrp^{\ell_0}\rightarrow\gmoduli$ from the set $\surfsubgrp^{\ell_0}$ to the moduli space $\gmoduli$ such that any $\Delta\in\surfsubgrp^{\ell_0}$ is mapped to the unique (projectivised) metric graph corresponding to any of its minimal carrier graphs.

We can thus define a family of probability measures lying on the moduli space $\gmoduli$ given by:
\begin{equation*}
    \gmeasl=\frac{1}{|\surfsubgrp^{\ell_0}(L)|}\sum_{\Delta\in\surfsubgrp^{\ell_0}(L)}\delta_{p(\Delta)}
\end{equation*}
for every $L>0$, where $\delta_{p(\Delta)}$ is the Dirac measure centred at $p(\Delta)$.

\begin{prop}\label{prop:limitmod}
    The family of measures $\gmeasl$ on $\gmoduli$ converge in the weak-*-topology to the probability measure $\meas$ as $L\to\infty$.
\end{prop}
\begin{proof}
    Given a trivalent graph $X$ with Euler-characteristic $\chi(X)=1-k$ denote by $\surfsubgrpx^{\ell_0}(L)$ the set of elements of $\surfsubgrp(L)$ with an $\ell_0$-long minimal carrier graph of type $X$. We then have that $\surfsubgrp^{\ell_0}(L)=\sum_{X\in\typ}\surfsubgrpx^{\ell_0}(L)$.

    \begin{claim}
        Fixing some trivalent $X$ then on the space $\gmoduli$ we have weak-*-convergence: 
        \begin{equation*}
        \gmeaslx:=\frac{1}{|\surfsubgrpx^{\ell_0}(L)|}\sum_{\Delta\in\surfsubgrpx^{\ell_0}(L)}\delta_{p(\Delta)}\to \frac{|\gaut|}{|\gtriv|}\sigma_{X}
    \end{equation*}
    as $L\to\infty$.
    \end{claim}
    \begin{proof}[Proof of Claim]
        Denote $\gmeasx:=\frac{|\gaut|}{|\gtriv|}\sigma_X$ and let $A\subset \gmoduli$ be a measurable subset with $\meas(\partial A)=0$. We will show that $\lim_{L\to\infty}\gmeaslx(A)=\gmeasx(A)$.

        Define a map $\pi:\critx^{\ell_0}\rightarrow S_X$ by $\pi(X,\phi)=\left[\frac{\ell_{\Sigma}(\phi(e_i))}{\ell_\Sigma(X,\phi)}\right]$. Corollary \ref{coroll:nummincar} tells us that for any $\Delta\in\surfsubgrpx^{\ell_0}$ there are $\left|\gaut\right|$ corresponding minimal carrier graphs in $\critx^{\ell_0}$ and these are such that $\pi(X,\phi)=p(\Delta)$. Thus we have:
        \begin{equation*}
            \begin{split}
                \gmeaslx(A)&=\frac{\left|\{\Delta\in\surfsubgrpx^{\ell_0}(L) : p(\Delta)\in A\}\right|}{|\surfsubgrpx^{\ell_0}(L)|} \\
                &=\frac{\left|\{(X,\phi)\in\critx^{\ell_0}(L): \pi(X,\phi)\in A\}\right|}{|\critx^{\ell_0}(L)|} \, .
            \end{split}
        \end{equation*}
        Denote by $\critAl(L) :=\{(X,\phi)\in\critx^{\ell_0}(L): \pi(X,\phi)\in A\}$ the set of $\ell_0$-long critical maps with length at most $L$ which project into the subset $A$.
        
        The cardinality of $\critAl(L)$ can be approximated by the sum of the number of critical maps in edge-length boxes with corners which project into $A$. Explicitly, for $h>0$ and $\vec{L}\in \R_+^{\mathbf{edge}(X)}$, let $\critx^{\ell_0}(\vec{L},h)$ denote the set of $\ell_0$-long critical maps $\phi:X\rightarrow\Sigma$ with $\ell_\Sigma(\phi(e))\in (L_e,L_e+h]$ for all $e\in\mathbf{edge}(X)$. For any $N\in\Z_{\geq0}$ let $\Omega(N):=\{\vec{n}\in\Z_{\geq0}^{3k-3} : ||\vec{n}||=N\}$ and $\Omega_A(N):=\{\vec{n}\in\Omega(N):\left[\frac{\vec{n}}{N}\right]\in A\}$. Then we have:
        \begin{equation*}
            \left|\critAl(Nh)\right| =\sum_{K=0}^N \,\sum_{\vec{n}\in\Omega_A(K)} \left|\critx^{\ell_0}(h\cdot\vec{n},h)\right| + \textbf{error}
        \end{equation*}
        where the error term arises from boxes whose projection does not lie entirely within $A$ and from boxes whose outer corner has $L^1$-norm greater than $Nh$, that being whose where $|\vec{n}|>N-E$. However, as $\meas(\partial A)=0$ and the number of edges $E$ is fixed, then these boundary boxes become negligible as $N\to\infty$. Thus for any $h>0$ and $\delta>0$ there exists $N'$ such that 
        \begin{equation*}
            \left|\critAl(Nh)\right| <(1+\delta)\sum_{K=0}^N \,\sum_{\vec{n}\in\Omega_A(K)} \left|\critx^{\ell_0}(h\cdot\vec{n},h)\right|
        \end{equation*}
        and
        \begin{equation*}
            \left|\critAl(Nh)\right| >(1-\delta)\sum_{K=0}^N \,\sum_{\vec{n}\in\Omega_A(K)} \left|\critx^{\ell_0}(h\cdot\vec{n},h)\right|
        \end{equation*}
        for all $N>N'$.
        
        A key element in the proof is the fact that generically the number of critical maps in a box depends only on the  $L^1$-norm of the corner of the box. Specifically, using \cite[Proposition~4.2]{ES23-commutator}, we have that:
        \begin{equation*}
            |\critx(\vec{L},h)|\sim \frac{2^{4\chi(X)}}{3^{3\chi(X)}}\cdot \pi^{\chi(X)}\cdot \frac{(e^h-1)^{-3\chi(X)}\cdot e^{||\vec{L}||}}{\operatorname{vol}(\Sigma)^{-\chi(X)}}
        \end{equation*}
        as $\min_{e\in\mathbf{edge(X)}}L_e\to\infty$, where $\chi(X)$ is the Euler-characteristic of the graph. Therefore there exists some $\ell_e>0$ depending only on $h$ and $\delta$ such that for any $\vec{L}\in \R_+^{\mathbf{edge}(X)}$ which is $\ell_e$-long we have
        \begin{equation*}
            |\critx(\vec{L},h)|< (1+\delta) \cdot c_{g,k}' \cdot(e^h-1)^{3k-3}\cdot e^{||\vec{L}||}
        \end{equation*}
        and 
        \begin{equation*}
            |\critx(\vec{L},h)|> (1-\delta) \cdot c_{g,k}'\cdot  (e^h-1)^{3k-3}\cdot e^{||\vec{L}||}
        \end{equation*}
        where to aid readability we denote:
        \begin{equation*}
            c'_{g,k}=\left(\frac{4}{3}\right)^{3-3k}\cdot (\pi^2(g-1))^{1-k} \, .
        \end{equation*}
        Lemma \ref{lem:shortnegl} tells us that the proportion of critical maps which are not $\ell_e$-long becomes negligible, thus for large enough $N$ we have that:
        \begin{equation*}
            \left|\critAl(Nh)\right| <(1+\delta)\cdot c_{g,k}'\cdot  (e^h-1)^{3k-3}\sum_{K=0}^N e^{Kh}\left|\Omega_A(K)\right|
        \end{equation*}
        and
        \begin{equation*}
            \left|\critAl(Nh)\right| >(1-\delta)\cdot c_{g,k}'\cdot  (e^h-1)^{3k-3}\sum_{K=0}^N e^{Kh}\left|\Omega_A(K)\right| .
        \end{equation*}
        Now using the definition of $\gmeasx$ we see that $\left|\Omega_A(K)\right|\sim \gmeasx(A) \left|\Omega(K)\right|$ as $K\to\infty$. We also have that $|\Omega(K)|\sim \frac{K^{3k-4}}{(3k-4)!}$ as $K\to\infty$. So for large enough $N$ we have that:
        \begin{equation*}
            \sum_{K=0}^N e^{Kh}\left|\Omega_A(K)\right|<(1+\delta) \frac{\gmeasx(A)}{(3k-4)!}\sum_{K=0}^N K^{3k-4}e^{Kh}
        \end{equation*}
        and 
        \begin{equation*}
            \sum_{K=0}^N e^{Kh}\left|\Omega_A(K)\right|>(1-\delta) \frac{\gmeasx(A)}{(3k-4)!}\sum_{K=0}^N K^{3k-4}e^{Kh} \, .
        \end{equation*}
        As $N$ grows we can approximate the sum in the above equations with an integral: 
        \begin{equation*}
            \begin{split}
                \sum_{K=0}^N K^{3k-4}e^{Kh}&=h^{3-3k}\sum_{K=0}^N (Kh)^{3k-4}e^{Kh}h \\
                &\sim h^{3-3k}\int_0^{Nh}x^{3k-4}e^x \, dx
            \end{split} 
        \end{equation*}
        and when $N\to\infty$ the integral value is asymptotic to $(Nh)^{3k-4}\cdot e^{Nh}$. Thus putting it all together we get:
        \begin{equation*}
            \left|\critAl(Nh)\right| <(1+\delta)\cdot \left(\frac{e^h-1}{h}\right)^{3k-3}\cdot \frac{c_{g,k}'\cdot (Nh)^{3k-4}\cdot e^{Nh}}{(3k-4)!}\cdot \gmeasx(A)
        \end{equation*}
        and 
        \begin{equation*}
            \left|\critAl(Nh)\right| >(1-\delta)\cdot \left(\frac{e^h-1}{h}\right)^{3k-3}\cdot \frac{c_{g,k}'\cdot (Nh)^{3k-4}\cdot e^{Nh}}{(3k-4)!}\cdot \gmeasx(A)
        \end{equation*}
        Combining Lemma \ref{lem:shortnegl} and Proposition \ref{prop:countcrit} we see that 
        \begin{equation*}
            |\critx^{\ell_0}(L)| \sim|\critx(L)| \sim \frac{c_{g,k}'\cdot L^{3k-4}\cdot e^{L}}{(3k-4)!} 
        \end{equation*}
        as $L\to\infty$, and thus for large enough $L$ we have:
        \begin{equation*}
            \begin{split}
                \gmeaslx(A)&=\frac{\left|\critAl(L)\right|}{|\critx^{\ell_0}(L)|} \\
                &< (1+\delta)\cdot \left(\frac{e^h-1}{h}\right)^{3k-3}\cdot \gmeasx(A)
            \end{split}
        \end{equation*}
        and 
        \begin{equation*}
            \gmeaslx(A)> (1-\delta)\cdot \left(\frac{e^h-1}{h}\right)^{3k-3}\cdot \gmeasx(A)
        \end{equation*}
        As we can take any value for $h$ and $\delta$ the result follows and we have proved the claim.
    \end{proof}
    
    Now from the definition of $\gmeasl$ we have that:
    \begin{equation*}
        \begin{split}
            \gmeasl(A)&=\frac{\left|\{\Delta\in\surfsubgrp^{\ell_0}(L): p(\Delta)\in A\}\right|}{\left|\surfsubgrp^{\ell_0}(L)\right|} \\
            &=\frac{1}{\left|\surfsubgrp^{\ell_0}(L)\right|}\sum_{X\in \typ}\left|\{\Delta\in\surfsubgrpx^{\ell_0}(L): p(\Delta)\in A\}\right| \\
            &=\frac{1}{\left|\surfsubgrp^{\ell_0}(L)\right|}\sum_{X\in \typ}\left|\surfsubgrpx^{\ell_0}(L)\right|\cdot \gmeaslx(A)
        \end{split}
    \end{equation*}
    for any $A\subset\gmoduli$ with $\meas(\partial A)=0$.
    Notice that Theorem \ref{thm:countingsubgrps} gives us that for any $X\in\typ$:
    \begin{equation*}
        \frac{\left|\surfsubgrpx^{\ell_0}(L)\right|}{\left|\surfsubgrp^{\ell_0}(L)\right|} \sim \frac{1}{|\gaut|\sum_{X' \in \typ}\frac{1}{|\operatorname{Aut}(X')|}}
    \end{equation*}
    as $L\to\infty$. Using this along with the result of the claim above we have that for any $\delta>0$ there exists some $L'>0$ such that for all $L>L'$
    \begin{equation*}
        \gmeasl(A)<(1+\delta)\frac{1}{\sum_{X \in \typ}\frac{1}{|\operatorname{Aut}(X)|}}\sum_{X\in\typ}\frac{1}{|\gtriv|}\sigma_X(A)
    \end{equation*}
    and 
    \begin{equation*}
        \gmeasl(A)>(1-\delta)\frac{1}{\sum_{X \in \typ}\frac{1}{|\operatorname{Aut}(X)|}}\sum_{X\in\typ}\frac{1}{|\gtriv|}\sigma_X(A)
    \end{equation*}
    Thus we have $\lim_{L\to\infty}\gmeasl(A)=\meas(A)$ and we are done.
\end{proof}

We now have what we need to prove our main result:
\begin{proof}[Proof of Theorem \ref{thm:limitmeasure}]
    As $\hsubgrps$ is locally compact then $\meas$ is the weak-*-limit of the family of measures $\measl$ if and only if we have
    \begin{equation*}
        \lim_{L\to\infty}\int f \,d\measl = \int f \,d\meas
    \end{equation*}
    for every continuous, compactly supported function $f\in C_c(\hsubgrps)$ \cite{ES-book}. 

    Proposition \ref{prop:limitmod} tells us that $\meas$ is the limit of the family of measures $\gmeasl$ on the moduli space $\gmoduli$. We can extend these measures to lie on the larger space $\hsubgrps\supset\gmoduli$ and the convergence still holds. Thus, fixing $\epsilon>0$, there exists some $L_0>0$ such that for all $L>L_0$:
    \begin{align*}
        \left| \int f \,d\measl - \int f \,d\meas\right| &\leq \left|\int f \,d\measl - \int f \,d\gmeasl \right|+ \left|\int f \,d\gmeasl- \int f \,d\meas \right| \\
        &< \left| \int f \,d\measl - \int f \,d\gmeasl \right| + \epsilon
    \end{align*}
    for all continuous, compactly supported functions $f\in C_c(\hsubgrps)$. 
    Thus it suffices to prove: 
    \begin{equation*}
        \left| \int f \,d\measl - \int f \,d\gmeasl \right| \to 0
    \end{equation*}
    as $L\to \infty$, for all such $f\in C_c(\hsubgrps)$. 

    Denote by $\surfsubgrp^{<\ell_0}(L):=\surfsubgrp(L)-\surfsubgrp^{\ell_0}(L)$ the subset of $\surfsubgrp(L)$ consisting of the subgroups which do not have an $\ell_0$-long carrier graph. From the definition of $\measl$ we can write:
    \begin{equation*}
        \begin{split}
            \int f \,d\measl&=\sum_{\Delta\in\surfsubgrp(L)}\frac{f(\Delta)}{|\surfsubgrp(L)|} \\
            &=\sum_{\Delta\in\surfsubgrp^{\ell_0}(L)}\frac{f(\Delta)}{|\surfsubgrp(L)|}+\sum_{\Delta\in\surfsubgrp^{<\ell_0}(L)}\frac{f(\Delta)}{|\surfsubgrp(L)|} \, .
        \end{split}
    \end{equation*}
    As $f$ is a compactly supported continuous function it is also bounded, thus there exists some $M>0$ such that $|f|<M$. Then using Lemma \ref{lem:shortnegl} there exists some $L'$ such that:
    \begin{equation*}
        \sum_{\Delta\in\surfsubgrp^{<\ell_0}(L)}\frac{|f(\Delta)|}{|\surfsubgrp(L)|} < \frac{M\cdot|\surfsubgrp^{<\ell_0}(L)|}{|\surfsubgrp(L)|} \, < \epsilon
    \end{equation*}
    for all $L>L'$.
    Applying the triangle inequality then gives:
    \begin{equation*}
        \begin{split}
            \left| \int f \,d\measl - \int f \,d\gmeasl \right|& =\left|\sum_{\Delta\in\surfsubgrp(L)}\frac{f(\Delta)}{|\surfsubgrp(L)|}-\sum_{\Delta\in\surfsubgrp^{\ell_0}(L)}\frac{f(p(\Delta))}{|\surfsubgrp^{\ell_0}(L)|}\right| \\
            & <\sum_{\Delta\in\surfsubgrp^{\ell_0}(L)}\left|\frac{f(\Delta)}{|\surfsubgrp(L)|}-\frac{f(p(\Delta))}{|\surfsubgrp^{\ell_0}(L)|}\right| + \epsilon
        \end{split}
    \end{equation*}
    for $L>L'$. We then again use Lemma \ref{lem:shortnegl} to get:
    \begin{equation*}
        \begin{split}
            \left|\frac{f(\Delta)}{|\surfsubgrp(L)|}-\frac{f(p(\Delta))}{|\surfsubgrp^{\ell_0}(L)|}\right|&=\left|\frac{|\surfsubgrp^{\ell_0}(L)|\cdot f(\Delta)-|\surfsubgrp(L)|\cdot f(p(\Delta))}{|\surfsubgrp(L)|\cdot |\surfsubgrp^{\ell_0}(L)|}\right| \\
            &\leq \frac{|\surfsubgrp^{\ell_0}(L)|\cdot |f(\Delta)-f(p(\Delta)|}{|\surfsubgrp(L)|\cdot |\surfsubgrp^{\ell_0}(L)|} + \frac{|\surfsubgrp^{<\ell_0}(L)|\cdot |f(p(\Delta))|}{|\surfsubgrp(L)|\cdot |\surfsubgrp^{\ell_0}(L)|} \\
            &< \frac{|f(\Delta)-f(p(\Delta)|}{|\surfsubgrp(L)|}+\frac{\epsilon}{|\surfsubgrp^{\ell_0}(L)|} 
        \end{split}
    \end{equation*}
    for any $\Delta\in\surfsubgrp^{\ell_0}(L)$ where $L>L'$. 
    
    \begin{claim}
        For any $f\in C_c(\hsubgrps)$ there exists some $L''>0$ such that 
        \begin{equation*}
            \left|f(\Delta) - f(p(\Delta))\right|<\epsilon
        \end{equation*}
        for all $\Delta\in\surfsubgrp^{\ell_0}$ with $\ell(\Delta)>L''$.
    \end{claim}
    \begin{proof}[Proof of Claim]
        Recall that the space $\hsubgrps$ has a quotient topology coming from a properly discontinuous $\out$-action on the space $\hreps\subset\projc$ (see Section \ref{sec:spaces}). We can lift the projection $p:\surfsubgrp^{\ell_0}\rightarrow\gmoduli$ to a projection $p':\critreps^{\ell_0}\rightarrow\cv$ such that a representation $\rho:\free\rightarrow\fund$ is mapped by $p'$ to the projectivised marked metric graph corresponding to the image of the minimal length critical graph map which induces $\rho$.

        We can also lift the function $f$ to a continuous function $f'\in C(\hreps)$ equivariant under the action of $\out$. 

        We know that as the length of representations in $\reps$ go to infinity their length functions limit to very small, minimal, $F_k$-equivariant isometric actions on $\R$-trees \cite{MS84}. More importantly, the non-trivial translation lengths of any faithful representation $\rho:\free\rightarrow\fund$ are bounded below by the systole of the surface $\Sigma$. Therefore the representations $\critreps$ limit to free, discrete actions on simplicial trees, i.e. points in Outer space $\cv$. Thus very long elements of $\critreps$ are `close' to their projections and specifically the continuity of $f'$ means that for any $\epsilon$ there exists some $L'$ such that
    \begin{equation*}
        \left| f'(\rho) - f'(p'(\rho)) \right|<\epsilon    
    \end{equation*}
    for any $\rho\in\critreps^{\ell_0}$ induced by a minimal length critical graph of length at least $L'$.

    It follows that in the quotient there also exists $L'>0$ such that
    \begin{equation*}
        \left| f(\Delta) - f(p(\Delta))\right|<\epsilon    
    \end{equation*}
    for any $\Delta\in\surfsubgrp^{\ell_0}$ with $\ell(\Delta)\geq L'$. We have proved the claim.
    \end{proof}
    
    From the claim it follows that for $L>L''$:
    \begin{equation*}
        \begin{split}
            \sum_{\Delta\in\surfsubgrp^{\ell_0}(L)}\frac{|f(\Delta)-f(p(\Delta)|}{|\surfsubgrp(L)|}&<\sum_{\substack{\Delta\in\surfsubgrp^{\ell_0}(L) \\
            \ell(\Delta)>L''}}\frac{\epsilon}{|\surfsubgrp(L)|}+\sum_{\Delta\in\surfsubgrp^{\ell_0}(L'')}\frac{2M}{|\surfsubgrp(L)|} \\
            &\leq\epsilon+\frac{2M\cdot|\surfsubgrp(L'')|}{|\surfsubgrp(L)|}
        \end{split}
    \end{equation*}
    
    Putting all this together we see that for any $L>\max\{L',L''\}$ we get:
    \begin{equation*}
        \begin{split}
            \left| \int f \,d\measl - \int f \,d\gmeasl \right|& <\sum_{\Delta\in\surfsubgrp^{\ell_0}(L)}\left(\frac{|f(\Delta)-f(p(\Delta)|}{|\surfsubgrp(L)|}+\frac{\epsilon}{|\surfsubgrp^{\ell_0}(L)|}\right) + \epsilon \\
            & \leq \frac{2M\cdot|\surfsubgrp(L'')|}{|\surfsubgrp(L)|} + 3\epsilon
        \end{split}
    \end{equation*}
    We then note that $2M\cdot|\surfsubgrp(L'')|$ is constant but $|\surfsubgrp(L)|$ grows exponentially with $L$. Therefore for large enough $L$ the first term above is itself less than $\epsilon$ and thus for such an $L$:
    \begin{equation*}
        \left| \int f \,d\measl - \int f \,d\meas\right|<5\epsilon \, .
    \end{equation*}
    We are finished. 
\end{proof}

The limit probability measure $\meas$ can be used to consider random properties of conjugacy classes of subgroups of $\fund$ and this also allows us to consider random covers of the surface. We will look at such properties in the Section \ref{sec:expectedprop}. 

\section{Patterson-Sullivan analogue}\label{sec:patt-sull}
The counting measures $\measl$ on $\hsubgrps$ studied in the section above are constructed by taking a finite subset $\surfsubgrp(L)\subset\surfsubgrp$ and weighting points in this subset uniformly. An alternative to this is a measure supported on the whole of the discrete set $\surfsubgrp$ where points corresponding to longer subgroups are less heavily weighted. By suitably controlling the weightings we are able to ensure that the resulting measure is a probability measure. 
The measures we construct in this section can be thought of as an analogue of Patterson-Sullivan measures \cite{Pat76}, however we note that there is no group action. 

\begin{lemma}
    The series
    \begin{equation*}
        p_s:=\sum_{\Delta\in \surfsubgrp}e^{-s\ell(\Delta)} \, .
    \end{equation*}
    converges for any $s>1$ and diverges when $s\leq1$. Thus we say that $s=1$ is the critical exponent.
\end{lemma}
\begin{proof}
    As the set $\surfsubgrp$ is infinite the series clearly diverges for any $s\leq 0$. Suppose then that $s>0$. 
    
    Denote by $N(t)=|\surfsubgrp(t)|$ the counting function for the number of elements of $\surfsubgrp$ with length at most $t\geq 0$. This is a step function and we can write the partial sums of the series as a Riemann–Stieltjes integral \cite{Stro20}:
    \begin{equation}\label{eq:ps-int}
    \begin{split}
        \sum_{\Delta\in \surfsubgrp(L)}e^{-s\ell(\Delta)}&=\int_0^L e^{-st}\, dN(t) \\
        &= e^{-sL}\cdot N(L)+ s\int_0^L e^{-st}\cdot N(t) \, dt 
    \end{split}
\end{equation}
    where we have used integration by parts to get the second equality.
    
    Theorem \ref{thm:countingsubgrps} tells us that for any $\delta>0$ there exists some $T_0>0$ such that 
    \begin{equation*}
        (1-\delta)\cdot c_{g,k}\cdot t^{3k-4}\cdot e^t<N(t)<(1+\delta)\cdot c_{g,k}\cdot t^{3k-4}\cdot e^t
    \end{equation*}
    for all $t>T_0$, where $c_{g,k}$ is given by: 
    \begin{equation*}
        c_{g,k}:=\sum_{X\in\typ}\frac{1}{\left|\gaut\right|} \cdot\left(\frac{4}{3}\right)^{3-3k}\cdot\frac{(\pi^2(g-1))^{1-k}}{(3k-4)!} \, .
    \end{equation*}
    Thus when $L>T_0$ we have that $e^{-sL}\cdot N(L)<(1+\delta)\cdot c_{g,k}\cdot L^{3k-4} \cdot e^{L(1-s)}$ and
    \begin{equation*}
        \begin{split}
            \int_0^L e^{-st}\cdot N(t) \, dt &= \int_0^{T_0} e^{-st}\cdot N(t) \, dt + \int_{T_0}^L e^{-st}\cdot N(t) \, dt \\
            &< \mathbf{const} + (1+\delta)\cdot c_{g,k}\int_{T_0}^L t^{3k-4}\cdot e^{t(1-s)} \, dt 
        \end{split}
    \end{equation*}
    where $\mathbf{const}$ is a constant depending only on $T_0$.
    
    When $s>1$ then $1-s$ is negative and therefore $L^{3k-4}\cdot e^{L(1-s)}\to 0$ as $L\to\infty$ and $t^{3k-4}\cdot e^{t(1-s)}$ is integrable on $[0,\infty)$. Thus:
    \begin{equation*}
        \begin{split}
            p_s&=\lim_{L\to\infty}\sum_{\Delta\in \surfsubgrp(L)}e^{-s\ell(\Delta)} \\
            &< s\cdot \mathbf{const} +s\cdot (1+\delta)\cdot c_{g,k}\int_{T_0}^\infty t^{3k-4}\cdot e^{t(1-s)} \, dt \\
            &<\infty
        \end{split}
    \end{equation*}
    and specifically $p_s$ converges.

    When $s\leq 1$ then 
    \begin{equation*}
        \begin{split}
            e^{-sL}\cdot N(L)&>(1-\delta)\cdot c_{g,k}\cdot L^{3k-4}\cdot e^{L(1-s)} \\
            &\geq (1-\delta)\cdot c_{g,k}\cdot L^{3k-4} \\
        \end{split}
    \end{equation*}
    and so the partial sums in (\ref{eq:ps-int}) tend to infinity as $L$ grows, thus $p_s$ diverges. 
\end{proof}

Now, for $s>1$ we can define a probability measure on $\hsubgrps$: 
\begin{equation*}
    \psmeas= \frac{1}{p_s}\sum_{\Delta\in\surfsubgrp}e^{-s\ell(\Delta)}\delta_\Delta
\end{equation*}
where $\delta_\Delta$ is the Dirac measure centred at $\Delta$.

\begin{thm:pattsulconverge}
    The family of measures $\psmeas$ on $\hsubgrps$ converge in the weak-*-topology to the limit measure $\meas$ as $s\downarrow1$.
\end{thm:pattsulconverge}
\begin{proof}
    Fix some $\epsilon>0$ and a continuous compactly supported function $f\in C_c(\hsubgrps)$. From Theorem \ref{thm:limitmeasure} we have that there exists some $L'>0$ such that 
    \begin{equation}\label{eq:ps-proof}
        \left|\int f \, d \,\measl - \int f \, d\meas \, \right| < \epsilon
    \end{equation}
    for any $L>L'$. 
    
    Denote $N_f(L):=\sum_{\Delta\in \surfsubgrp(L)} f(\Delta)$. Then using the Riemann-Stieltjes integral and integration by parts as above gives:
    \begin{equation*}
        \begin{split}
            \int f \, d\,\psmeas &= \frac{1}{p_s}\sum_{\Delta\in\surfsubgrp}e^{-s\ell(\Delta)}f(\Delta) \\
            &=\frac{1}{p_s}\int_0^\infty e^{-st} \, d N_f(t) \\
            &= \frac{s}{p_s}\int_{L'}^\infty e^{-st}\cdot N_f(t) \, dt +\frac{s}{p_s}\int_0^{L'}e^{-st}\cdot N_f(t) \, dt \, .
        \end{split}
    \end{equation*}
    Consider the contribution from the finite interval:
    \begin{equation*}
        \left|\frac{s}{p_s}\int_0^{L'}e^{-st}\cdot N_f(t) \, dt \right|< \frac{s}{p_s}\int_0^{L'}\left| e^{-t}\cdot N_f(t) \right| \, dt \, .
    \end{equation*}
    Noting that $p_s\to\infty$ as $s\downarrow 1$ and $\int_0^{L'}\left| e^{-t}\cdot N_f(t) \right| \, dt$ is a constant independent of $s$ we see there exists some $s'>1$ such that
    \begin{equation}\label{eq:ps1}
        \left|\frac{s}{p_s}\int_0^{L'}e^{-st}\cdot N_f(t) \, dt \right|< \epsilon
    \end{equation}
    for all $s\in (1,s')$.

    From the definition of $p_s$ and utilising the Riemann-Stieltjes integral as in (\ref{eq:ps-int}) above we have:
    \begin{equation*}
        p_s = s\int_0^{L'}e^{-st}\cdot N(t) \, dt + s\int_{L'}^\infty e^{-st}\cdot N(t) \, dt \, .
    \end{equation*}
    Then, denoting $\mathbf{F}:=\int f \, d\meas$ and noting that this is also a constant independent of $s$, we can find some $s''>0$ such that
    \begin{equation}\label{eq:ps2}
        \begin{split}
            \frac{1}{p_s}\left| \mathbf{F} \cdot s\int_0^{L'}e^{-st}\cdot N(t) \, dt \right| &< \frac{|\mathbf{F}|\cdot s}{p_s}\int_0^{L'}e^{-t}\cdot N(t) \, dt \\
            &< \epsilon
        \end{split}
    \end{equation}
    for all $s\in(1,s'')$.

    From the definition of $\measl$ we have that $N_f(L)=N(L)\cdot\int f \, d\measl$ and so (\ref{eq:ps-proof}) can be written as 
    \begin{equation}\label{eq:ps3}
        \left|N_f(t) - N(t)\cdot\mathbf{F}\right|<\epsilon \cdot N(t)
    \end{equation}
    for any $t>L'$.
    
    Putting together (\ref{eq:ps1}), (\ref{eq:ps2}), (\ref{eq:ps3}) and the triangle inequality we get:
    \begin{equation*}
        \begin{split}
            \left|\int f \, d \,\psmeas - \int f \, d\meas\right| &< \frac{1}{p_s} \left|s\int_{L'}^\infty e^{-st}\cdot N_f(t) \, dt - p_s \cdot \mathbf{F} \right| +\epsilon \\
            &<\frac{s}{p_s} \int_{L'}^\infty e^{-st}\cdot \left|N_f(t) - N(t)\cdot\mathbf{F}\right| \, dt  +2\epsilon \\
            &< \frac{\epsilon \cdot s}{p_s} \int_{L'}^\infty e^{-st} \cdot N(t) \, dt  +2\epsilon  \\
            &=\epsilon\cdot\frac{\int_{L'}^\infty e^{-st}\cdot N(t)  \, dt}{\int_0^\infty e^{-st}\cdot N(t)\, dt} + 2\epsilon\\
            &< 3\epsilon
        \end{split}
    \end{equation*}
    for any $s\in(1,\min\{s',s''\})$. We are done.
\end{proof}

\section{Expected properties of covers}\label{sec:expectedprop}
For an element $\Delta\in\surfsubgrp$ we will denote by $\cover=\hyp/\Delta$ the corresponding cover of $\Sigma$. 
Theorem \ref{thm:limitmeasure} allows us to determine asymptotically expected geometric and topological properties of a cover picked uniformly at random from the set 
\begin{equation*}
    \cov(L)=\{\cover \mid \Delta\in\surfsubgrp(L)\}
\end{equation*}
as $L\to\infty$. 
In this section we present three examples of properties we are able to study and give explicit calculations for $k=2$.

\subsection*{Systole of a random cover}
The systole of a surface is the length of its shortest non-trivial closed geodesic. To find the expected value of the systole of a random cover $\cover\in\cov$ we can calculate the expected systole of a rank $k$ metric graph using our probability measure on the moduli space $\gmoduli$. 

More explicitly, for any $\Delta\in\subgrps$ we denote by $\sys$ the systole of the surface $\cover=\hyp/\Delta$. Then for any $L>0$ let
\begin{equation*}
    \mathbb{E}_{\operatorname{sys}}(L):=\frac{1}{|\surfsubgrp(L)|}\sum_{\Delta\in\surfsubgrp(L)}\sys
\end{equation*}
be the expected value of the systole for the set $\cov(L)$.

\begin{prop}\label{prop:expsysbylimit}
    For any $X\in\gmoduli$ let $f_{\operatorname{sys}}(X)$ equal the systole of the metric graph $X$, then we have:
    \begin{equation*}
        \mathbb{E}_{\operatorname{sys}}(L) \sim L \cdot \int_{\gmoduli} f_{\operatorname{sys}} \, d\meas
    \end{equation*}
    as $L\to\infty$.
\end{prop}
\begin{proof}
We know from Theorem \ref{thm:countingsubgrps} that the number of elements $|\surfsubgrp(L)|$ grows exponentially in $L$. Thus for any $\epsilon$ we have that the proportion of elements of $|\surfsubgrp(L)|$ with lengths greater than $(1-\epsilon)L$ tends to 1 as $L\to\infty$. Therefore we have:
\begin{equation*}
    \begin{split}
        \frac{\mathbb{E}_{\operatorname{sys}}(L)}{L}&=\frac{1}{|\surfsubgrp(L)|}\sum_{\Delta\in\surfsubgrp(L)}\frac{\sys}{L} \\
        &\to\frac{1}{|\surfsubgrp(L)|}\sum_{\Delta\in\surfsubgrp(L)}\frac{\sys}{\ell(\Delta)}
    \end{split}
\end{equation*}
as $L\to\infty$.

The function $f_{\operatorname{sys}}$ on $\gmoduli$ sends a element $X\in\gmoduli$ to the systole of the metric graph $X$. We can extend this function continuously to the whole space $\hsubgrps$ where $f_{\operatorname{sys}}:\hsubgrps\rightarrow\R$ is defined by $f_{\operatorname{sys}}(\Delta)=\frac{\sys}{\ell(\Delta)}$ for $\Delta\in\subgrps$ where $\ell(\Delta)$, as always, is the length of some shortest carrier graph of $\Delta\in\subgrps$. 
From the definition of $\measl$ we have: 
\begin{equation*}
    \frac{1}{|\surfsubgrp(L)|}\sum_{\Delta\in\surfsubgrp(L)}\frac{\sys}{\ell(\Delta)}=\int_{\hsubgrps}f_{\operatorname{sys}} \, d\measl
\end{equation*}
for any $L>0$.
Then as $f_{\operatorname{sys}}$ is bounded, Theorem \ref{thm:limitmeasure} gives us:
\begin{equation*}
    \int_{\hsubgrps}f_{\operatorname{sys}} \, d\measl\to\int_{\hsubgrps}f_{\operatorname{sys}} \, d\meas
\end{equation*}
as $L\to \infty$.
\end{proof}

Using Proposition \ref{prop:expsysbylimit} allows us to make explicit calculations for the expected value of the systole. We do this here for $k=2$:

\begin{coroll}\label{coroll:systole}
     The expected value $\mathbb{E}_{\operatorname{sys}}(L)$ for the set of covers of the surface $\Sigma$ with fundamental group isomorphic to $F_2$ is asymptotic to $\frac{23}{90}L$ as $L\to\infty$.
\end{coroll}
\begin{proof}
For rank $k=2$ there are only two topological types of trivalent graphs, these are the dumbbell graph $X_\text{db}$ (the two vertices each have a loop and the vertices are connected by 1 edge) and the theta graph $X_\theta$ (the two vertices are connected by 3 edges). See Figure \ref{fig:rank2graphs}.

\begin{figure}[ht]
    \centering
\begin{tikzpicture}[scale=1]
    \node at (0.75,-1) {$X_{\text{db}}$};
  \node[draw, circle, inner sep=1pt, fill] (v1) at (0,0) {};
  \node[draw, circle, inner sep=1pt, fill] (v2) at (1.5,0) {};
  \draw (-0.7,0) circle [radius=0.7];
  \node at (-1.5,0.5) {$e_1$};
  \draw (2.2,0)  circle [radius=0.7];
  \node at (3,0.5) {$e_2$};
  \draw (v1) -- (v2) node[midway,above] {$e_3$};

    \node at (5,-1) {$X_{\theta}$};
  \node[draw, circle, inner sep=1pt, fill] (v3) at (5,0) {};
  \node[draw, circle, inner sep=1pt, fill] (v4) at (7.2,0) {};
  \draw (v3) -- (v4) node[midway,above] {$e_2$};
  \path (v3) edge[out=80,in=100, looseness=1.6] node[midway,above] {$e_1$} (v4);
  \path (v3) edge[out=-80,in=-100, looseness=1.6] node[midway,above] {$e_3$} (v4);
\end{tikzpicture}
\caption{The two rank-2 trivalent graphs: the dumbbell graph and the theta graph.}
    \label{fig:rank2graphs}
\end{figure}
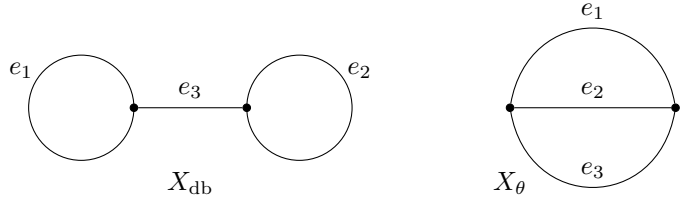
The dumbbell graph has a total of 8 different automorphisms $|\operatorname{Aut}(X_{\text{db}})|=8$. Four of these act trivially on all edges lengths, thus taking $\gtriv$ as defined in (\ref{def:triv}) above we get $|\operatorname{Triv}(X_{\text{db}})|=4$.
The theta graph has a total of 12 different automorphisms $|\operatorname{Aut}(X_{\theta})|=12$ and two of these act trivially $|\operatorname{Triv}(X_{\theta})|=2$.

Therefore, with $\sigma_X$ as in (\ref{def:sigma-meas}), the limit measure for $k=2$ is:
\begin{align*}
    \mathbf{m}_{2} &= \frac{1}{\sum_{X\in\operatorname{Typ}_{2}}\frac{1}{\gaut}} \cdot \sum_{X\in\operatorname{Typ}_{2}}\frac{1}{|\gtriv|}\sigma_{X} \\
    &= \frac{24}{5}\left(\frac{1}{4}\cdot\sigma_{X_{\text{db}}}+\frac{1}{2}\cdot\sigma_{X_{\theta}}\right) \\
    &=\frac{6}{5}\cdot\sigma_{X_{\text{db}}}+\frac{12}{5}\cdot\sigma_{X_{\theta}}
\end{align*}

The systole of a dumbbell graph is the length of the shortest edge making a loop. So, taking $e_1$, $e_2$ and $e_3$ to be the lengths of the edges as labelled in Figure \ref{fig:rank2graphs} we have the systole is equal to 
\begin{equation*}
    \operatorname{sys}(X_{\text{db}})=\min\{e_1,e_2\} \, .
\end{equation*}
However, note that there is a non-trivial automorphism of the dumbbell which flips the graph swapping the edges $e_1$ and $e_2$. Thus in moduli space the metric dumbbell graph $X$ with edge lengths $(e_1,e_2,e_3)=(x,y,z)$ is equivalent to the metric dumbbell graph $X'$ with edge lengths $(e_1',e_2',e_3')=(y,x,z)$. This means that any point in moduli space corresponding to a dumbbell graph can be uniquely determined by a triple $(x,y,z)$ with $x\geq y$.
Given that the graphs in moduli space have volume one, we can parametrise the metric dumbbell graphs in moduli space by $\{x,y\in\R_+:x+y<1,x\geq y\}$ and we note that here $f_{\operatorname{sys}}=y$.
Recalling the definition of the $\sigma_X$ measure from (\ref{def:sigma-meas}) we then get: 
\begin{equation*}
    \begin{split}
        \int_{\hat{\mathcal{G}_k}} f_{\operatorname{sys}} \, d \sigma_{X_{\text{db}}} &=2\int_0^{\frac{1}{2}}\int_y^{1-y} y \, dx\, dy \\
        &=2\int_0^{\frac{1}{2}} y(1-2y)\, dy \\
        &=\frac{1}{12} \, .
    \end{split}
\end{equation*}

The systole of a theta graph is the sum of the lengths of the two shortest edges. So, again with labelling from Figure \ref{fig:rank2graphs} we have that the systole is equal to 
\begin{equation*}
    \operatorname{sys}(X_\theta)=\min\{e_1+e_2,e_1+e_3,e_2+e_3\} \, .
\end{equation*}
The theta graph has automorphisms which permute the edges of the graph. Thus a metric theta graph $X$ with edge lengths $(e_1,e_2,e_3)=(x,y,z)$ is equivalent in moduli space to any metric theta graph $X'$ with the same set of edge lengths $\{e_1',e_2',e_3'\}=\{x,y,z\}$ and we can uniquely determine a point in moduli space corresponding to a theta graph with a triple $(x,y,z)$ with $z\geq x\geq y$.
We can therefore parametrise the metric theta graphs in moduli space by $\{x,y\in\R_+: y\leq x\leq 1-x-y\}$ and here $f_{\operatorname{sys}}=x+y$. We then get:
\begin{equation*}
    \begin{split}
        \int_{\hat{\mathcal{G}_k}} f_{\operatorname{sys}} \, d \sigma_{X_{\theta}} &=2\int_0^{\frac{1}{3}}\int_y^{\frac{1-y}{2}} x+y \, dx\, dy \\
        &=2\int_0^{\frac{1}{3}} \frac{1}{8}+\frac{1}{4}y -\frac{15}{8}y^2 \, dy \\
        &=\frac{7}{108} 
    \end{split}
\end{equation*}
Putting this all together then gives us:
\begin{equation*}
    \begin{split}
        \mathbb{E}_{\operatorname{sys}}(L)&\sim L\int_{\hat{\mathcal{G}_{2}}}f_{\operatorname{sys}}\, d\mathbf{m}_{2} \\
        &= \frac{6}{5}L\int_{S_{X_\text{db}}} f_{\operatorname{sys}} \, d\sigma_{X_\text{db}} + \frac{12}{5}L\int_{S_{X_\theta}} f_{\operatorname{sys}} \, d\sigma_{X_\theta} \\
        &=\frac{6}{5}\cdot \frac{1}{12} L + \frac{12}{5}\cdot \frac{7}{108} L \\
        &=\frac{23}{90}L
    \end{split}
\end{equation*}
as $L\to\infty$.
\end{proof}
Note that Theorem \ref{thm:pattsulconverge} tells us that if we instead take the expected value to be defined by $\psmeas$ then we get the same asymptotic value as above when $s\downarrow1$. 

\subsection*{Separating orthogeodesics on random covers}
On hyperbolic surfaces much work has been done on studying the properties of random closed geodesics. 
Notably, a celebrated result of Mirzakhani \cite{Mirz-simple} gives that on a closed genus $2$ hyperbolic surface the probability that a simple closed geodesic is separating is $\frac{1}{49}$.

Here, rather than study random closed curves we investigate orthogeodesics on the convex core of random covers. Recall that an orthogeodesic on a hyperbolic surface with geodesic boundary is a geodesic arc joining boundary components which meet the boundary perpendicularly. An orthogeodesic is separating if it splits the surface into two connected components.

Here investigate the probability that the convex core of a random cover of a surface has a separating orthogeodesic of bounded length. As mentioned above we have that in the rank $k=2$ case there are only two types of trivalent graphs: the dumbbell graph and the theta graph. We note that the dumbbell graph has a separating edge but the theta graph does not. Therefore we may expect that covers coming from the dumbbell are easier to separate. This is the case, in fact generically a cover corresponding to a dumbbell graph will have a `short' separating orthogeodesic whereas a cover corresponding to a theta graph will not. We make this explicit here:

\begin{coroll}\label{coroll:orthgeo}
    For any $\lambda>0$, the probability that the convex core of a cover $\cover\in\mathbf{Cov}_{\Sigma,2}(L)$ has a separating orthogeodesic of length less than $\lambda$ tends to $\frac{3}{5}$ as $L\to\infty$.
\end{coroll}
\begin{proof}
    The convex core of a cover $\cover\in\mathbf{Cov}_{\Sigma,2}(L)$ with a dumbbell minimal length carrier graph is a pair of pants with geodesic boundary. Using the labelling from Figure \ref{fig:rank2graphs}, we see that two of the boundary components correspond to the two loops formed by $e_1$ and $e_2$, the third boundary component corresponds to the loop travelling around the whole graph: $e_1+e_3+e_2+e_3$. As $L\to\infty$ the length of the boundary components tends towards the length of the respective loops in the graph. Thus this third boundary component has length approximately $L+e_3$. There is a orthogeodesic perpendicular to the separating edge $e_3$ which joins this long boundary component to itself and separates the surface, see Figure \ref{fig:orthogeodesic}. Given two more orthogeodesics, specifically the non-separating orthogeodesics perpendicular to the edge loops $e_1$ and $e_2$, we can split the pair of pants into two hyperbolic hexagons. Then using right-angled hexagon formulas \cite{Bus92} we can show that the length of the separating orthogeodesic decreases exponentially as the length of the outer boundary component grows. Thus for any $\lambda>0$ we can find some $L'$ such that all covers corresponding to a minimal length dumbbell carrier graph of length at least $L'$ have a separating orthogeodesic of length less than $\lambda$.
\begin{figure}[ht]
    \centering
    \includegraphics[width=0.7\linewidth]{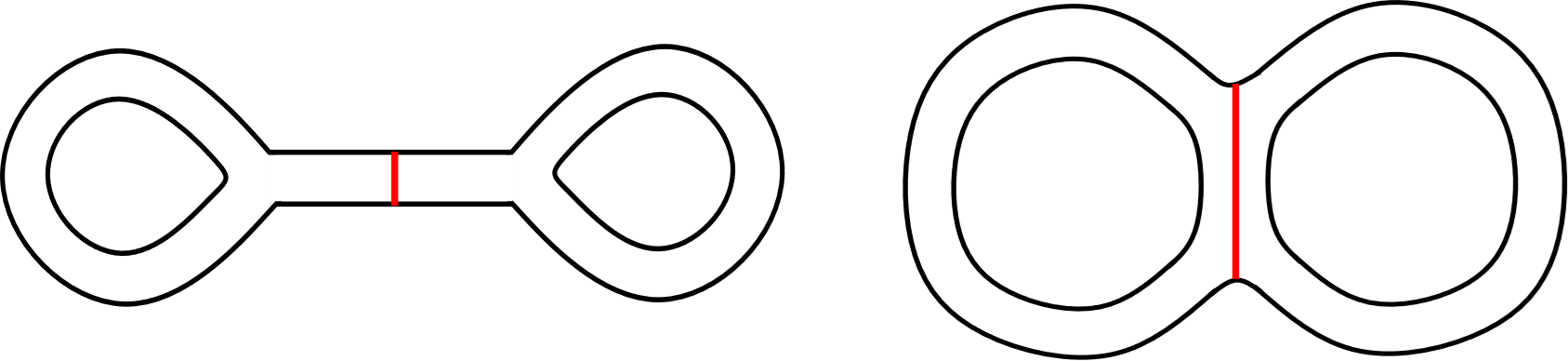}
    \caption{Pair of pants from a dumbbell graph and a theta graph with shortest separating orthogeodesic shown in red}
    \label{fig:orthogeodesic}
\end{figure}
    
    A cover corresponding to a theta graph can either be a pair of pants or a one-holed torus. If the cover is a one-holed torus then it does not have any separating orthogeodesics. If the cover is a pair of pants then any separating orthogeodesic must follow along an edge of the graph and thus the length of a separating orthogeodesic can be bounded below by the length of the shortest edge. Lemma \ref{lem:shortnegl} tells us that for any $\lambda>0$ the proportion of elements of $\mathbf{G}_{\Sigma,2}(L)$ which are not $\lambda$-long tends to $0$ as $L\to\infty$. So generically a pair of pants corresponding to a theta graph will not have a separating orthogeodesic of length less that $\lambda$.

    Therefore the probability that a cover $\cover\in\mathbf{Cov}_{\Sigma,2}(L)$ has a separating orthogeodesic of length less than $\lambda$ tends to the probability that a element $\Delta\in\mathbf{G}_{\Sigma,2}(L)$ has a dumbbell minimal carrier graph. 
    We can use Theorem \ref{thm:countingsubgrps} to calculate asymptotically the proportion of elements of $\surfsubgrp(L)$ with a dumbbell minimal carrier graph:
    \begin{equation}\label{eq:propotiondb}
        \begin{split}
            \frac{|\mathbf{G}_{X_{\text{db}},\Sigma}(L)|}{|\mathbf{G}_{\Sigma,2}(L)|} 
            &\sim\frac{1}{|\operatorname{Aut}(X_\text{db})|}\cdot\frac{1}{\frac{1}{|\operatorname{Aut}(X_\text{db})|}+\frac{1}{|\operatorname{Aut}(X_\theta)|}}\\
            &=\frac{3}{5}
        \end{split}
    \end{equation}
    as $L\to\infty$, where we have used that $|\operatorname{Aut}(X_\text{db})|=8$ and $|\operatorname{Aut}(X_\theta)|=12$.
\end{proof}

\subsection*{Topological types of random covers}
Another property that is interesting to study is the expected topological types of the covers of a surface. 

In the $k=2$ case, there are two possible topological types of covers: the one-holed torus (a genus 1 surface with 1 boundary component) and the pair of pants (a genus 0 surface with 3 boundary component). See Figure \ref{fig:rank2surfaces}.
\begin{figure}[ht]
    \centering
    \includegraphics[width=0.7\linewidth]{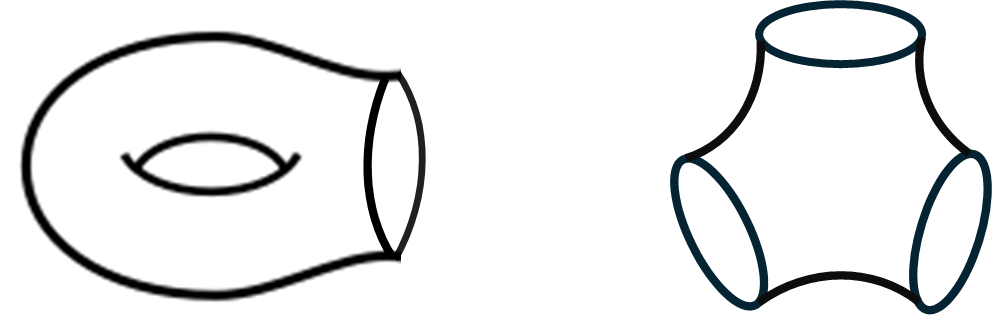}
    \caption{The two rank-2 surfaces: the one-holed torus and the pair of pants.}
    \label{fig:rank2surfaces}
\end{figure}

\begin{coroll}\label{coroll:toptype}
    Asymptotically, the probability that a cover $\cover\in\mathbf{Cov}_{\Sigma,2}(L)$ is a pair of pants is greater than or equal to $\frac{3}{5}$ as $L\to\infty$.
\end{coroll}
\begin{proof}
    Recall again that there are two types of trivalent rank $2$ graphs: the dumbbell graph $X_\text{db}$ and the theta graph $X_\theta$. Any thickening of the dumbbell graph will give a pair of pants whereas a thickening of theta graph can be either a pair of pants or a one-holed torus. Thus the probability that a cover $\cover$ of an element $\Delta\in\mathbf{G}_{\Sigma,2}(L)$ is a pair of pants is bounded below by the probability that a minimal carrier graph of $\Delta$ is a dumbbell graph.

    The calculation  (\ref{eq:propotiondb}) above gives us that asymptotically this is equal to $\frac{3}{5}$ as $L\to\infty$
\end{proof}

As mentioned in the introduction, a forthcoming paper by the author \cite{Wri} studies in more detail the topological properties of random covers. This allows us to make precise calculations for the probability of topological types of random covers and give an exact value for Corollary \ref{coroll:toptype}. 

\printbibliography

\end{document}